\documentclass{amsart}
\usepackage{tikz, hyperref,amssymb,amsmath}
\usetikzlibrary{decorations.markings}
\usepackage[mathscr]{eucal}

\newtheorem{thm}{Theorem}[section]
\newtheorem{cor}[thm]{Corollary}
\newtheorem{lem}[thm]{Lemma}
\newtheorem{prop}[thm]{Proposition}

\theoremstyle{definition}

\theoremstyle{remark}
\newtheorem{rmk}[thm]{Remark}

\newtheorem{example}[thm]{Example}

\numberwithin{equation}{section}

\newcommand{\CC}{\mathbb{C}}
\newcommand{\NN}{\mathbb{N}}
\newcommand{\QQ}{\mathbb{Q}}
\newcommand{\TT}{\mathbb{T}}
\newcommand{\ZZ}{\mathbb{Z}}
\newcommand{\RR}{\mathbb{R}}

\newcommand{\Cc}{\mathcal{C}}
\newcommand{\Gg}{\mathcal{G}}
\newcommand{\Hh}{\mathcal{H}}
\newcommand{\Ii}{\mathcal{I}}

\newcommand{\Mm}{\mathcal{M}}
\newcommand{\Oo}{\mathcal{O}}
\newcommand{\Pp}{\mathcal{P}}
\newcommand{\Qq}{\mathcal{Q}}

\newcommand{\sA}{\mathscr{A}}
\newcommand{\sB}{\mathscr{B}}

\newcommand{\Ad}{\operatorname{Ad}}
\newcommand{\Aut}{\operatorname{Aut}}
\newcommand{\clsp}{\overline{\lsp}}
\newcommand{\dom}{\operatorname{dom}}
\newcommand{\id}{\operatorname{id}}
\newcommand{\Ind}{\operatorname{Ind}}
\newcommand{\lsp}{\operatorname{span}}
\newcommand{\lt}{\operatorname{lt}}

\newcommand{\Per}{\operatorname{Per}}
\newcommand{\Prim}{\operatorname{Prim}}
\newcommand{\shift}{T}

\newcommand{\Zcat}[1]{\underline{Z}^{#1}}

\title[Simplicity of twisted $k$-graph algebras]{Simplicity of twisted $C^*$-algebras of higher-rank graphs and crossed products by quasifree actions}

\author{Alex Kumjian}
\address[A. Kumjian]{Department of Mathematics (084)\\ University
of Nevada\\ Reno NV 89557-0084\\ USA} \email{alex@unr.edu}

\author{David Pask}
\address[D.Pask and A. Sims]{School of Mathematics and Applied Statistics\\
University of Wollongong\\
NSW  2522\\
AUSTRALIA}
\email{dpask, asims@uow.edu.au}
\author{Aidan Sims}

\thanks{This research was supported by the Australian Research Council. We thank Becky Armstrong for bringing a number of typographical errors to our attention.}

\date{\today}

\begin{document}
\dedicatory{Dedicated to George A.\ Elliott on the occasion of his 70th birthday.}
\begin{abstract}
We characterise simplicity of twisted $C^*$-algebras of row-finite $k$-graphs with no
sources. We show that each 2-cocycle on a cofinal $k$-graph determines a canonical
second-cohomology class for the periodicity group of the graph. The groupoid of the
$k$-graph then acts on the cartesian product of the infinite-path space of the graph with
the dual group of the centre of any bicharacter representing this second-cohomology
class. The twisted $k$-graph algebra is simple if and only if this action is minimal. We
apply this result to characterise simplicity for many twisted crossed products of
$k$-graph algebras by quasifree actions of free abelian groups.
\end{abstract}

\maketitle

\section{Introduction}
Higher-rank graphs, or $k$-graphs, are combinatorial objects introduced by the first two
authors in \cite{KP2000} as graph-based models for the higher-rank Cuntz--Krieger
algebras studied by Robertson and Steger \cite{RobSte:JraM99}. These $C^*$-algebras have
been widely studied \cite{BNR, DavidsonYang:CJM09, Evans:NYJM08, FarthingMuhlyEtAl:SF05},
and their fundamental structure theory is by now fairly well understood. They provide
interesting examples in noncommutative geometry \cite{PaskRennieEtAl:CMP09,
SkalskiZacharias:HJM08} and have been used to establish weak semiprojectivity
\cite{Spielberg:JOT07} and calculate nuclear dimension \cite{RuizSimsEtAl:xx14} for UCT
Kirchberg algebras.

Recently \cite{KPS3}, we introduced a cohomology theory for $k$-graphs, and studied the
associated twisted $k$-graph algebras $C^*(\Lambda, c)$ \cite{KPS4,KPS5}. These include
many examples that do not arise naturally from untwisted $k$-graph $C^*$-algebras (see
\cite[Example~7.10]{KPS4}, \cite{PaskSierakowskiEtAl:xx14}, and
Theorem~\ref{thm:easytwistsimplicity} below). So twisted $k$-graph $C^*$-algebras could
serve as useful models of various classes of classifiable $C^*$-algebras, particularly as
they are always nuclear and belong to the UCT class \cite[Corollary~8.7]{KPS4}.

It is therefore important to understand when twisted $k$-graph $C^*$-algebras are simple.
Simplicity of untwisted $k$-graph algebras was characterised in \cite{RoSi}, and
\cite[Corollary~8.2]{KPS4} shows that if $C^*(\Lambda)$ is simple, so is every
$C^*(\Lambda, c)$. But the converse fails: regarding $\NN^2$ as a $2$-graph $T_2$, the
associated $2$-graph algebra $C^*(T_2) \cong C(\TT^2)$ is not simple, but for each
irrational number $\theta$ there is a cocycle $c_\theta$ such that $C^*(T_2, c_\theta)$
is the irrational-rotation algebra $A_\theta$. This elementary example indicates that
characterising simplicity of $C^*(\Lambda, c)$ is a subtle problem---an indication
confirmed by the partial results in \cite{SWW}. In this paper, we present a complete
solution to this problem: Theorem~\ref{thm:main simple} gives a necessary and sufficient
condition for $C^*(\Lambda, c)$ to be simple.

The broad strokes of our solution are as follows. We show that $C^*(\Lambda, c)$ is
isomorphic to the $C^*$-algebra of a Fell bundle $\sB_\Lambda$ over a topologically
principal amenable \'etale quotient $\Hh_\Lambda$ of the $k$-graph groupoid introduced in
\cite{KP2000}. Results of Ionescu and Williams \cite{IW} show that if $\sB$ is Fell
bundle over a groupoid $\Hh$, then $\Hh$ acts on the primitive ideal space of the
$C^*$-subalgebra $C^*(\Hh^{(0)}; \sB) \subseteq C^*(\Hh; \sB)$ sitting over the unit
space of $\Hh$. We prove that if $\Hh$ is topologically principal, amenable and \'etale,
then $C^*(\Hh; \sB)$ is simple if and only if this action is minimal. In particular,
$C^*(\Hh_\Lambda; \sB_\Lambda)$ is simple if and only if the Ionescu-Williams action of
$\Hh_\Lambda$ is minimal. Our task is then to identify the action; it turns out that this
is quite intricate.

The $k$-graph has a periodicity group $\Per(\Lambda) \subseteq \ZZ^k$ \cite{CKSS}. We use
groupoid technology to show that $c$ determines a class $[\omega] \in H^2(\Per(\Lambda),
\TT)$. We then identify a map from $\Hh_\Lambda$ to the primitive ideal space of the
noncommutative torus $A_\omega$ that becomes a homomorphism when $\Prim(A_\omega)$ is
regarded as a quotient of $\TT^k$. The map $\Hh_\Lambda \to \Prim(A_\omega)$ determines
an action $\theta$ of $\Hh_\Lambda$ on $\Hh_\Lambda^{(0)} \times \Prim(A_\omega)$. We
prove that there is a homeomorphism between this cartesian-product space and
$\Prim(C^*(\Hh^{(0)}_\Lambda; \sB_\Lambda))$ which intertwines the action $\theta$ with
the action described in the preceding paragraph, and deduce our main result.

Even the action $\theta$ is not easy to compute in a generic example. So we develop some
key examples where it \emph{can} be computed. Firstly, if $A_\omega$ is simple, then
$\theta$ is minimal precisely when $\Lambda$ is cofinal. So we show how to decide
efficiently whether $A_\omega$ is simple. Secondly, we show how to recover many twisted
crossed products of $k$-graph algebras by quasifree actions of $\ZZ^l$ as twisted
$(k+l)$-graph algebras, and show that our main theorem yields a very satisfactory
characterisation of simplicity of many such twisted crossed products.

\medskip

The paper is organised as follows. In Section~\ref{sec:background} we establish the
background that we need regarding $k$-graphs and their groupoids. This includes a fairly
general technical result giving a sufficient condition for the interior of the isotropy
of a Hausdorff \'etale groupoid to be closed. The main point is that the interior of the
isotropy in the groupoid $\Gg_\Lambda$ of a cofinal $k$-graph $\Lambda$ is always closed,
and relatively easy to describe: Corollary~\ref{cor:ILambda} says that if $\Per(\Lambda)$
is the periodicity group of the $k$-graph discussed in \cite{CKSS}, then the interior
$\Ii_\Lambda$ of the isotropy in $\Gg_\Lambda$ can be identified canonically with
$\Gg_\Lambda^{(0)} \times \Per(\Lambda)$. Moreover, we can form the quotient groupoid
$\Hh_\Lambda := \Gg_\Lambda/\Ii_\Lambda$, and this quotient is itself a topologically
principal amenable \'etale groupoid.

In Section~\ref{sec:action and main}, we study carefully the relationship between
cocycles on a $k$-graph $\Lambda$, and dynamics associated to the corresponding $k$-graph
groupoid. In \cite{KP2000}, the $k$-graph algebra $C^*(\Lambda)$ is identified with the
$C^*$-algebra of a groupoid $\Gg_\Lambda$ with unit space $\Lambda^\infty$, the space of
infinite paths in $\Lambda$. We showed in \cite{KPS4} that each $C^*(\Lambda, c)$ can be
realised as a twisted $C^*$-algebra $C^*(\Gg_\Lambda, \sigma_c)$ of the same groupoid.
The restriction of the groupoid cocycle $\sigma_c$ to each fibre of $\Ii_\Lambda$
determines a 2-cocycle of $\Per(\Lambda)$. We show that these 2-cocycles are all
cohomologous, and deduce in Proposition~\ref{prp:constant in Ii} that $\sigma_c$ is
cohomologous to a cocycle $\sigma$ whose restriction to $\Ii_\Lambda$ is of the form
$1_{\Lambda^\infty} \times \omega$ for some bicharacter $\omega$ of $\Per(\Lambda)$. We
then have $C^*(\Lambda, c) \cong C^*(\Gg_\Lambda, \sigma)$, and we seek to characterise
simplicity of the latter. As in \cite{OPT}, the primitive-ideal space of each fibre of
$C^*(\Ii_\Lambda, \sigma)$ can be identified with the character space of the kernel
$Z_\omega \subseteq \Per(\Lambda)$ of the antisymmetric bicharacter associated to
$\omega$. So $\Prim C^*(\Ii_\Lambda, \sigma)$ can be identified with $\Lambda^\infty
\times \widehat{Z}_\omega$. We construct from $\sigma$ a
$\Per(\Lambda)\widehat{\;}$-valued 1-cocycle $r^\sigma$ on $\Gg_\Lambda$. We then have
the wherewithal to state our main theorem, Theorem~\ref{thm:main simple}, although the
proof must wait until the end of the subsequent section. We show in
Lemma~\ref{lem:compute omega} how to compute the bicharacter $\omega$, and hence the
group $Z_\omega$ appearing in the main theorem, without passing to the groupoid
$\Gg_\Lambda$. We then show that $r^\sigma$ determines an action $\theta$ of the quotient
$\Hh_\Lambda$ of $\Gg_\Lambda$ by the interior of its isotropy on $\Lambda^\infty \times
\widehat{Z}_\omega$. Theorem~\ref{thm:main simple} can be recast as saying that
$C^*(\Lambda, c)$ is simple if and only if $\theta$ is minimal (see
Corollary~\ref{cor:simple<->minimal}).

In Section~\ref{sec:Fell bundle} we prove Theorem~\ref{thm:main simple} using the
technology of Fell bundles. We use ideas from \cite{DKR} (see also \cite{Ram99}) to
recover $C^*(\Lambda, c)$ as the $C^*$-algebra of a Fell bundle $\sB_\Lambda$ over the
quotient $\Hh_\Lambda$ of $\Gg_\Lambda$ discussed above. We show that the restriction
$C^*(\Hh_\Lambda^{(0)}; \sB_\Lambda)$ of $C^*(\Hh_\Lambda; \sB_\Lambda)$ to the unit
space of $\Hh_\Lambda$ is isomorphic to $C_0(\Lambda^\infty) \otimes C^*(\Per(\Lambda),
\omega)$. Since $\Per(\Lambda)$ is a free abelian group, the $C^*$-algebra
$C^*(\Per(\Lambda), \omega)$ is a noncommutative torus, and has primitive ideal space
$\Lambda^\infty \times \widehat{Z}_\omega$ (see, for example, \cite{Rieffel:PSPM82}).
Results of Ionescu and Williams \cite{IW} show that conjugation in $\sB_\Lambda$
determines an action of $\Hh_\Lambda$ on the primitive-ideal space of
$C^*(\Hh_\Lambda^{(0)}; \sB_\Lambda)$ and hence on $\Lambda^\infty \times
\widehat{Z}_\omega$. By identifying specific elements in the fibres of $\sB_\Lambda$ that
implement the Ionescu--Williams action, we prove that it matches up with the action
$\theta$ of Section~\ref{sec:action and main}. In Lemma~\ref{lem:bundle uniqueness} and
Corollary~\ref{cor:minimality}, we adapt the standard argument in \cite{Ren} to prove
that if $\sB$ is a Fell bundle over a topologically principal amenable \'etale groupoid
$\Hh$, then $C^*(\Hh; \sB)$ is simple if and only if the Ionescu-Williams action of $\Hh$
is minimal. We then prove our main theorem by applying this result to the bundle
$\sB_\Lambda$ over $\Hh_\Lambda$.

In Section~\ref{sec:CP}, we investigate a broad class of examples of twisted higher-rank
graph $C^*$-algebras where the hypotheses of our main result are readily checkable. We
show how a $\TT^l$-valued $1$-cocycle on a $k$-graph combined with a bicharacter $\omega$
of $\ZZ^l$ can be combined to obtain a 2-cocycle on the Cartesian-product $(k+l)$-graph
$\Lambda \times \NN^l$ for which the associated twisted $(k+l)$-graph $C^*$-algebra is
isomorphic to a twisted crossed product of $C^*(\Lambda)$ by $\ZZ^l$. We demonstrate that
the hypotheses of our main theorem can be effectively checked for these examples, and
obtain a usable characterisation of simplicity of crossed products arising in this way
when $\Lambda$ is aperiodic.

We finish in Section~\ref{sec:examples} by presenting a number of concrete examples of
our main result, showing how all of its working parts interact and demonstrating that
each of the ingredients of the statement is genuinely necessary to obtain a satisfactory
characterisation of simplicity. We also present a somewhat unrelated example which we
believe is nevertheless interesting in its own right: a $3$-graph all of whose twisted
$C^*$-algebras (including the untwisted one) are simple, but for which the twisted
$C^*$-algebras are not all mutually isomorphic.

\section{Background}\label{sec:background}

Throughout the paper, we regard $\NN^k$ as a monoid under addition, with identity 0 and
generators $e_1, \dots, e_k$. For $m,n \in \NN^k$, we write $m_i$ for the
$i$\textsuperscript{th} coordinate of $m$, and define $m \vee n \in \NN^k$ by $(m \vee
n)_i = \max\{m_i, n_i\}$.

Given a small category $\Cc$, we write $\Cc^{*2} = \{(\lambda, \mu) \in \Cc \times \Cc :
s(\lambda) = r(\mu)\}$ for the collection of composable pairs in $\Cc$. A $\TT$-valued
$2$-cocycle\footnote{In \cite{KPS4} these were called \emph{categorical cocycles}, in
contradistinction to cubical cocycles.} $c$ on $\Cc$ is a map $c :\Cc^{*2} \to \TT$ such
that $c(\lambda, s(\lambda)) = c(r(\lambda), \lambda) = 1$ for all $\lambda$ and
$c(\mu,\nu)c(\lambda, \mu\nu) = c(\lambda,\mu)c(\lambda\mu, \nu)$ for composable
$\lambda$, $\mu$, $\nu$. If $b : \Cc \to \TT$ is a function with $b(\alpha) = 1$ for
every identity morphism $\alpha$ of $\Cc$, then $\delta^1 b(\mu,\nu) :=
b(\mu)b(\nu)\overline{b(\mu\nu)}$ defines a $2$-cocycle called the 2-coboundary
associated to $b$. Two $2$-cocycles $c,c'$ are cohomologous if $(\mu,\nu) \mapsto
\overline{c(\mu,\nu)}c'(\mu,\nu)$ is a 2-coboundary.

If $\Cc$ is a discrete group, then any function $c : \Cc \times \Cc \to \TT$ for which the functions
$c(\cdot, \alpha)$ and $c(\beta, \cdot)$ are homomorphisms is a cocycle;
such cocycles are called \emph{bicharacters.}
If, in addition, $\Cc$ is abelian, every $2$-cocycle is cohomologous to a bicharacter
(see \cite[Theorem 7.1]{Kleppner}).
For more background on 2-cocycles and bicharacters on abelian groups see \cite{Kleppner, BK, OPT}
(note that in the first two references 2-cocycles are called multipliers).

\subsection{\texorpdfstring{$k$}{k}-graphs and twisted
\texorpdfstring{$C^*$}{C*}-algebras}

Let $\Lambda$ be a countable small category and $d : \Lambda \to \NN^k$ be a functor.
Write $\Lambda^n := d^{-1}(n)$ for $n \in \NN^k$. Then $\Lambda$ is a $k$-graph (see
\cite{KP2000}) if $d$ satisfies the \emph{factorisation property}: $(\mu,\nu) \mapsto
\mu\nu$ is a bijection of $\{(\mu,\nu) \in \Lambda^m \times \Lambda^n : s(\mu) =
r(\nu)\}$ onto $\Lambda^{m+n}$ for each $m,n \in \NN^k$. We then have $\Lambda^0 =
\{\id_o : o\in \operatorname{Obj}(\Lambda)\}$, and so we regard the domain and codomain
maps as maps $s, r : \Lambda \to \Lambda^0$. Recall from \cite{PaskQuiggEtAl:NYJM04} that
for $v,w \in \Lambda^0$ and $X \subseteq \Lambda$, we write
\begin{gather*}
v X := \{\lambda \in X : r(\lambda) = v\},\qquad X w := \{\lambda \in X : s(\lambda) =
w\},\qquad\text{and}\\
v X w = vX \cap Xw.
\end{gather*}
A $k$-graph $\Lambda$ is \emph{row-finite with no sources} if $0 < |v\Lambda^n| < \infty$
for all $v \in \Lambda^0$ and $n \in \NN^k$. See \cite{KP2000} for further details
regarding the basic structure of $k$-graphs.

Let $\Lambda$ be a row-finite $k$-graph with no sources. We recall the definition of the
\emph{infinite path space} $\Lambda^\infty$ given in \cite[Definition 2.1]{KP2000}. We
write $\Omega_k$ for the $k$-graph $\{(m,n) \in \NN^k : m \le n\}$ with $r(m,n) = (m,m)$,
$s(m,n) = (n,n)$, $(m,n)(n,p) = (m,p)$ and $d(m,n) = n - m$. We identify $\Omega_k^{0}$
with $\NN^k$ by $(m,m) \mapsto m$. We define $\Lambda^\infty$ to be the collection of all
$k$-graph morphisms $x : \Omega_k \to \Lambda$. For $p \in \NN^k$, we define $\shift^p :
\Lambda^\infty \to \Lambda^\infty$ by $(\shift^p x)(m,n) := x(m+p, n+p)$ for all $(m,n)
\in \Omega_k$. (Traditionally, as in \cite{KP2000}, these shift maps $\shift^p$ have been
denoted $\sigma^p$, but we will use $\sigma$ in this paper to denote a $2$-cocycle on
$\Gg_\Lambda$.) For $x \in \Lambda^\infty$ we denote $x(0)$ by $r(x)$. For $\lambda \in
\Lambda$ and $x \in \Lambda^\infty$ with $r(x) = s(\lambda)$, there is a unique element
$\lambda x \in \Lambda^\infty$ such that $(\lambda x)(0, d(\lambda)) = \lambda$ and
$\shift^{d(\lambda)}(\lambda x) = x$.

As in \cite[Definition~4.7]{KP2000}, we say that $\Lambda$ is \emph{cofinal} if, for
every $x \in \Lambda^\infty$ and every $v \in \Lambda^0$, there exists $n \in \NN^k$ such
that $v \Lambda x(n) \not= \emptyset$. We say that $\Lambda$ is \emph{aperiodic} if it
satisfies the ``aperiodicity condition" of \cite[Definition~4.3]{KP2000}: for every $v
\in \Lambda^0$ there exists $x \in \Lambda^\infty$ with $r(x) = v$ and $\shift^m(x) \not=
\shift^n(x)$ whenever $m \not= n$.

Given a $k$-graph $\Lambda$, the group of all $2$-cocycles on $\Lambda$ (as described
above) is denoted $\Zcat2(\Lambda, \TT)$. Let $\Lambda$ be a row-finite $k$-graph with no
sources, and fix $c \in \Zcat2(\Lambda,\TT)$. A Cuntz--Krieger $(\Lambda,c)$-family in a
$C^*$-algebra $B$ is a function $t : \lambda \mapsto t_\lambda$ from $\Lambda$ to $B$
such that
\begin{itemize}
\item[(CK1)] $\{t_v : v \in \Lambda^0\}$ is a collection of
    mutually orthogonal projections;
\item[(CK2)] $t_\mu t_\nu = c(\mu,\nu)t_{\mu\nu}$ whenever
    $s(\mu) = r(\nu)$;
\item[(CK3)] $t^*_\lambda t_\lambda = t_{s(\lambda)}$ for all $\lambda \in \Lambda$; and
\item[(CK4)] $t_v = \sum_{\lambda \in v\Lambda^n} t_\lambda t^*_\lambda$ for all $v \in \Lambda^0$ and $n \in \NN^k$.
\end{itemize}

$C^*(\Lambda, c)$ is then defined to be the universal $C^*$-algebra generated by a
Cuntz--Krieger $(\Lambda,c)$-family (see \cite[Notation 5.4]{KPS4}).

\subsection{Groupoids}

A groupoid is a small category $\Gg$ with inverses. We use standard groupoid notation as
in, for example, \cite{Ren}. So $\Gg^{(0)}$ is the set of identity morphisms of $\Gg$,
called the unit space, and $\Gg^{(2)}$ denotes the set $\Gg^{*2}$ of composable pairs in
$\Gg$. The groupoid $\Gg$ is an \'etale Hausdorff groupoid if it has a locally compact
Hausdorff topology under which all operations in $\Gg$ are continuous (when $\Gg^{(2)}
\subseteq \Gg \times \Gg$ is given the relative topology) and the range and source maps
$r,s : \Gg \to \Gg^{(0)}$ are local homeomorphisms. It then makes sense to talk about
continuous cocycles on $\Gg$. We write $Z^2(\Gg, \TT)$ for the group of continuous
$\TT$-valued 2-cocycles on $\Gg$ and say that two continuous 2-cocycles are cohomologous
if they differ by a continuous 2-coboundary---that is, the coboundary $\delta^1b$
associated to a continuous map $b : \Gg \to \TT$ such that $b|_{\Gg^{(0)}} \equiv 1$. A
$1$-cocycle on $\Gg$ with values in a group $G$ is a map $\rho : \Gg \to G$ that carries
composition in $\Gg$ to the group operation in $G$. Given $u \in \Gg^{(0)}$ we write
$\Gg^u$ for $\{\gamma \in \Gg : r(\gamma) = u\}$, $\Gg_u$ for $\{\gamma \in \Gg :
s(\gamma) = u\}$ and $\Gg^u_u = \Gg^u \cap \Gg_u$. The \emph{isotropy} of $\Gg$ is the
set $\bigcup_{u \in \Gg^{(0)}} \Gg^u_u$ of elements of $\Gg$ whose range and source
coincide. A groupoid is \emph{minimal} if $r(\Gg_u)$ is dense in $\Gg^{(0)}$ for every
unit $u \in \Gg^{(0)}$. It is \emph{topologically principal} if $\{u \in \Gg^{(0)} :
\Gg^u_u = \{u\}\}$ is dense in $\Gg^{(0)}$.

It will be important later to know that the interior of the isotropy in the groupoid
associated to a cofinal $k$-graph is closed. This will follow from the following fairly
general result.

\begin{prop}\label{prp:intiso closed}
Let $\Gg$ be an \'etale groupoid, let $G$ be a countable discrete abelian group, and let
$c \in Z^1(\Gg, G)$. Suppose that $\Gg$ is minimal and that for all $x$, the restriction
of $c$ to $\Gg^x_x$ is injective. Let $\Ii$ denote the interior of the isotropy of $\Gg$.
For $x,y \in \Gg^{(0)}$, we have $c(\Ii \cap \Gg^x_x) = c(\Ii \cap \Gg^y_y)$. The set $H$
defined by $H := c(\Ii \cap \Gg^x_x)$ for any $x \in \Gg^{(0)}$ is a subgroup of $G$, and
$s \times c$ induces an isomorphism from $\Ii$ to $\Gg^{(0)} \times H$. In particular the
interior of the isotropy of $\Gg$ is closed.
\end{prop}
\begin{proof}
For $x \in \Gg^{(0)}$ set $\Ii_x := \Ii \cap \Gg^x_x$ and note that $H_x := c(\Ii_x)$ is
a subgroup of $G$. Fix $x, y \in \Gg^{(0)}$; we prove that $H_x = H_y$. By symmetry it
suffices to show that $H_x \subset H_y$, so we fix $h \in H_x$ and prove that $h \in
H_y$. Fix $\alpha \in \Ii_x$ such that $c(\alpha) = h$. Since $G$ is discrete and $c$ is
continuous, there is an open neighbourhood $U$ of $\alpha$ such that $U \subseteq \Ii
\cap c^{-1}(h)$. Since $c$ is injective on each $\Gg^x_x$, this $U$ is a bisection. Since
$\Gg$ is minimal, the set $s(\Gg^y)$ is dense in $\Gg^{(0)}$, and so there exists $\gamma
\in \Gg^y$ such that $s(\gamma) \in s(U)$. The unique element $\beta$ of $\Ii_{s(\gamma)}
\cap U$ satisfies $c(\beta) = h$, and then $\gamma \beta \gamma^{-1} \in \Ii_y$ satisfies
$c(\gamma\beta\gamma^{-1}) = c(\gamma) c(\beta) \overline{c(\gamma)} = h$. So $h \in H_y$
as required.

Thus the map $s \times c$ yields an isomorphism from the interior of the isotropy of
$\Gg$ to $\Gg^{(0)} \times H$.   Since $\Ii$ is the intersection of the closed set
$c^{-1}(H)$ with the isotropy of $\Gg$, which is also closed, we deduce that $\Ii$ is
closed.
\end{proof}

Given  an \'etale groupoid $\Gg$ and a 2-cocycle $\sigma \in Z^2(\Gg, \TT)$, it is
straightforward to check that $C_c(\Gg)$ is a *-algebra under the operations
\[
(fg)(\gamma) = \sum_{\eta\zeta = \gamma}
\sigma(\eta,\zeta)f(\eta)g(\zeta)
\qquad\text{ and }\qquad
f^*(\gamma) = \overline{\sigma(\gamma, \gamma^{-1})} \overline{f(\gamma^{-1})}
\]
for $f, g \in C_c(\Gg)$.
The twisted groupoid $C^*$-algebra  $C^*(\Gg, \sigma)$ is then defined to be the closure
of $C_c(\Gg)$ under the maximal $C^*$-norm (see \cite{Ren} for more details).

\subsection{\texorpdfstring{$k$}{k}-graph groupoids}\label{sec:kgrgpds}

Following \cite[Definition 2.7]{KP2000} we associate a groupoid $\Gg_\Lambda$ to each
row-finite $k$-graph $\Lambda$ with no sources by putting
\[
\Gg_\Lambda :=
    \{(x, l - m, y) \in \Lambda^\infty \times \ZZ^k \times \Lambda^\infty : l, m \in \NN^k, \shift^l x = \shift^m y\}.
\]
For $\mu, \nu \in \Lambda$ with $s(\mu) = s(\nu)$ define $Z(\mu, \nu) \subset
\Gg_\Lambda$ by
\[
Z(\mu, \nu) := \{ (\mu x, d(\mu) - d(\nu), \nu x) : x \in \Lambda^\infty, r(x) = s(\mu)\}.
\]
For $\lambda \in \Lambda$, we define $Z(\lambda) := Z(\lambda, \lambda)$.

The sets $Z(\mu, \nu)$ form a basis of compact open sets for a locally compact Hausdorff
topology on $\Gg_\Lambda$ under which it is an \'{e}tale groupoid \noindent with
structure maps $r (x,l-m,y) = (x,0,x)$, $s ( x , l-m,y)= (y,0,y)$, and $(x,l-m,y)
(y,p-q,z) = (x , l-m+p-q,z)$. (see \cite[Proposition 2.8]{KP2000}). The $Z(\lambda)$ are
then a basis for the relative topology on $\Gg_\Lambda^{(0)} \subseteq \Gg_\Lambda$. We
identify $\Gg_\Lambda^{(0)} = \{ (x,0,x) : x \in \Lambda^\infty \}$ with
$\Lambda^\infty$. Following \cite[Proposition 2.8, Corollary 3.5]{KP2000} $\Gg_\Lambda$
is an amenable \'etale groupoid. Moreover $\Gg_\Lambda$ is minimal if and only if
$\Lambda$ is cofinal \cite[Proof of Proposition~4.8]{KP2000}.

Suppose now that $\Lambda$ is cofinal. As in \cite{CKSS}, we define a relation on
$\Lambda$ by $\mu \sim \nu$ if and only if $s( \mu ) = s( \nu )$ and $\mu x = \nu x$ for
all $x \in s(\mu) \Lambda^\infty$. This is an equivalence relation on $\Lambda$ which
respects range, source and composition. By \cite[Theorem 4.2(1)]{CKSS}, the set $\Per(
\Lambda ) := \{ d( \mu ) - d( \nu ) : \mu , \nu \in \Lambda \text{ and } \mu \sim \nu \}
\subseteq \ZZ^k$ is a subgroup of $\ZZ^k$. Since $\Lambda$ is cofinal,
\cite[Lemma~4.6]{CKSS} gives
\[
\Per(\Lambda) \subseteq \{m \in \ZZ^k : (x, m, x) \in \Gg_\Lambda\text{ for all } x \in \Lambda^\infty\}.
\]
Since $\Per(\Lambda)$ is a subgroup of $\ZZ^k$, it is also a finitely generated free
abelian group and so $\Per(\Lambda) \cong \ZZ^l$ for some integer $l \le k$.

\begin{cor}\label{cor:ILambda}
Let $\Lambda$ be a cofinal row-finite $k$-graph with no sources. Let $\Ii_\Lambda$ denote
the interior of the isotropy in $\Gg_\Lambda$. Then $\Ii_\Lambda$ is closed and
\[
\Ii_\Lambda = \{(x, m , x) : x \in \Lambda^\infty, m \in \Per(\Lambda)\} \cong \Lambda^\infty \times \Per(\Lambda).
\]
Moreover, $\Hh_\Lambda := \Gg_\Lambda/\Ii_\Lambda$ is
an amenable, topologically principal, locally compact, Hausdorff, \'etale groupoid.
\end{cor}
\begin{proof}
Note that $\Gg_\Lambda$ is a minimal \'etale groupoid and the restriction of the
canonical cocycle $c \in Z^1(\Gg_\Lambda, \ZZ^k)$, given by $c(x,n,y) = n$ to
$(\Gg_\Lambda)_x^x$ is injective for each $x \in \Lambda^\infty$. Hence by
Proposition~\ref{prp:intiso closed} $\Ii_\Lambda$ is closed.

By definition of the topology on $\Gg_\Lambda$, the set $\{(x, m , x) : x \in
\Lambda^\infty, m \in \Per(\Lambda)\}$ is contained in  $\Ii_\Lambda$, the interior of
the isotropy of $\Gg_\Lambda$. Conversely, if $\alpha \in \Ii_\Lambda$, then $\alpha =
(x,m,x)$ for some $x \in \Lambda^\infty$ and $m \in \ZZ^k$, and
Proposition~\ref{prp:intiso closed} applied to the cocycle $c$ of the preceding paragraph
shows that $(y,m,y) \in \Ii_\Lambda$ for every $y$. In particular $m \in \Per(\Lambda)$.
So $\Ii_\Lambda = \{(x, m , x) : x \in \Lambda^\infty, m \in \Per(\Lambda)\}$ as claimed.

It is routine to check that $\Gg_\Lambda/\Ii_\Lambda$ is a locally compact Hausdorff
\'etale groupoid (see, for example, \cite[Proposition~2.5]{SimWil:xxxx}). Since
$c(\Ii_\Lambda) = \Per(\Lambda)$, there exists $\tilde{c} \in Z^1(\Hh_\Lambda,
\ZZ^k/\Per(\Lambda))$ such that $\tilde{c}([(x, n, y)]) = n + \Per(\Lambda)$. The
groupoid $\tilde{c}^{-1}(0)$ is isomorphic to $c^{-1}(0)$ which is amenable by, for
example, \cite[Lemma~6.7]{Yeend:jot07}. Since $\ZZ^k/\Per(\Lambda)$ is abelian and hence
amenable, it follows from \cite[Proposition 9.3]{Speilberg:TAMS14} that $\Hh_\Lambda$ is
amenable. By construction of $\Ii_\Lambda$, the interior of the isotropy of $\Hh_\Lambda$
is trivial, and therefore $\Hh_\Lambda$ is topologically principal by, for example,
\cite[Proposition~3.6]{Renault:imsb08}.
\end{proof}

We frequently identify $\Ii_\Lambda$ with $\Lambda^\infty \times \Per(\Lambda)$ as in
Corollary~\ref{cor:ILambda}.

Let $\Lambda \mathbin{{_s*_s}} \Lambda := \{(\mu,\nu) \in \Lambda \times \Lambda : s(\mu)
= s(\nu)\}$. Lemma~6.6 of \cite{KPS4} shows that there exists $\Pp \subseteq \Lambda
\mathbin{{_s*_s}} \Lambda$ such that
\begin{equation}\label{eq:Pproperties}
    (\lambda, s(\lambda)) \in \Pp\text{ for all $\lambda$}\quad\text{ and }\quad
        \Gg_\Lambda = \bigsqcup_{(\mu,\nu) \in \Pp} Z(\mu,\nu).
\end{equation}
Given such a set $\Pp$, for $g \in \Gg_\Lambda$ we write $(\mu_g, \nu_g)$ for the element
of $\Pp$ with $g \in Z(\mu_g,\nu_g)$. Fix $c \in \Zcat2(\Lambda, \TT)$. Let $\tilde{d} :
\Gg_\Lambda \to \ZZ^k$ be the canonical 1-cocycle $\tilde{d}(x,n,y) = n$. Lemma~6.3 of
\cite{KPS4} shows that for composable $g,h \in \Gg_\Lambda$, there exist $\alpha, \beta,
\gamma \in \Lambda$ and $y \in \Lambda^\infty$ such that
\begin{equation}\label{eq:alphabetagamma}
\begin{split}
&\nu_g \alpha = \mu_h \beta,\ \; \mu_g\alpha = \mu_{gh} \gamma
    \;\text{ and }\; \nu_h\beta = \nu_{gh}\gamma; \text{ and} \\
g = (\mu_g\alpha y, \tilde{d}(g), &\nu_g\alpha y),\ \;
    h = (\mu_h\beta y, \tilde{d}(h), \nu_h\beta y)\;\text{ and }\;
    gh = (\mu_{gh}\gamma y, \tilde{d}(gh), \nu_{gh}\gamma y).
\end{split}
\end{equation}
The formula
\begin{equation}\label{eq:sigmac}
 \sigma_c(g,h)
    = c(\mu_g, \alpha)\overline{c(\nu_g, \alpha)}
        c(\mu_h, \beta)\overline{c(\nu_h, \beta)}
        \overline{c(\mu_{gh},\gamma)}c(\nu_{gh},\gamma)
\end{equation}
does not depend on the choice of $\alpha,\beta,\gamma$, and determines a continuous
groupoid $2$-cocycle on $\Gg_\Lambda$. If $\sigma'_c$ is obtained from $c$ in the same
way with respect to another collection $\Pp'$, then $\sigma_c$ and $\sigma'_c$ are
cohomologous. Corollary~7.7 of \cite{KPS4} shows that there is an isomorphism
$C^*(\Lambda, c) \cong C^*(\Gg_\Lambda, \sigma_c)$ that carries each $s_\lambda$ to the
characteristic function $1_{Z(\lambda, s(\lambda))}$.

\section{An action of the \texorpdfstring{$k$}{k}-graph groupoid associated to a
\texorpdfstring{$k$}{k}-graph 2-cocycle}\label{sec:action and main}

We consider a row-finite $k$-graph $\Lambda$ with no sources. Lemma~7.2 of \cite{SWW}
says that if $\Lambda$ is not cofinal, then $C^*(\Lambda, c)$ is nonsimple for every $c
\in \Zcat2(\Lambda, \TT)$. Since we are interested here in when $C^*(\Lambda, c)$
\emph{is} simple, we shall therefore assume henceforth that $\Lambda$ is cofinal.

Recall that if $\sigma$ is a continuous $\TT$-valued 2-cocycle on a groupoid $\Gg$, then
there is a groupoid extension $\Gg^{(0)} \times \TT \to \Gg \times_\sigma \TT \to \Gg$,
where $\Gg \times_\sigma \TT$ is equal to $\Gg \times \TT$ as a set, with unit space
$(\Gg \times_\sigma \TT)^{(0)} = \Gg^{(0)} \times \{ 1 \}$, range map $r(g, z) = (r(g),
1)$, source map $s(g, z) = (s(g), 1)$ and operations
\[
(\alpha, w)(\beta, z) = (\alpha\beta, \sigma(\alpha, \beta)zw)\quad\text{and}\quad
(\alpha, w)^{-1} = (\alpha^{-1}, \overline{\sigma(\alpha, \alpha^{-1})} \overline{w}).
\]
Given $\sigma \in  Z^2(\Gg_\Lambda, \TT)$ we write $\Ii_\Lambda \times_\sigma \TT$ for
$\Ii_\Lambda \times \TT$ regarded as a subgroupoid of $\Gg_\Lambda \times_\sigma \TT$. We
often implicitly identify $(\Gg_\Lambda \times_\sigma \TT)^{(0)}$ with
$\Gg_\Lambda^{(0)}$.

\begin{prop}\label{prp:constant in Ii}
Let $\Lambda$ be a cofinal row-finite $k$-graph with no sources, and take $c \in
\Zcat2(\Lambda, \TT)$. Fix $\Pp \subseteq \Lambda \mathbin{{_s*_s}} \Lambda$ as
in~\eqref{eq:Pproperties}, and let $\sigma_c \in Z^2(\Gg_\Lambda, \TT)$ be as
in~\eqref{eq:sigmac}. For $x \in \Lambda^\infty$, define $\sigma_c^x \in
Z^2(\Per(\Lambda), \TT)$ by $\sigma^x_c(p,q) = \sigma_c((x,p,x),(x,q,x))$. Then there is
a bicharacter $\omega$ of $\Per(\Lambda)$ such that $\sigma^x_c$ is cohomologous to
$\omega$ for every $x \in \Lambda^\infty$. For any such bicharacter $\omega$, there
exists a cocycle $\sigma \in Z^2(\Gg_\Lambda, \TT)$ such that $\sigma$ is cohomologous to
$\sigma_c$ and $\sigma|_{\Ii_\Lambda} = \omega \times 1_{\Lambda^\infty}$.
\end{prop}

To prove the proposition, we first prove some lemmas.

\begin{lem}\label{lem:r-alpha}
Let $\Lambda$ be a row-finite $k$-graph with no sources. Take $\sigma \in
Z^2(\Gg_\Lambda, \TT)$. For $\alpha = (x,m,y) \in \Gg_\Lambda$ and $p \in \Per(\Lambda)$,
define $r^\sigma_\alpha : \Per(\Lambda) \to \TT$ by
\begin{equation}\label{eq:conjugation}
r^\sigma_\alpha(p)	
	= \sigma\big(\alpha,(y,p,y)\big)\sigma\big((x,m+p,y), \alpha^{-1}\big)
		\overline{\sigma\big(\alpha,\alpha^{-1}\big)}.
\end{equation}
For $x \in \Lambda^\infty$ define $\sigma_x \in Z^2(\Per(\Lambda), \TT)$ by
$\sigma_x(p,q) = \sigma((x,p,x), (x,q,x))$. For each $p \in \Per(\Lambda)$, the map
$r^\sigma_{\cdot}(p) : \Gg_\Lambda \to \TT$ is continuous, and we have
\begin{equation}\label{eq:rsigmaformula}
r^\sigma_\alpha(p+q)
    = \sigma_{r(\alpha)}(p,q)\overline{\sigma_{s(\alpha)}(p,q)} r^\sigma_\alpha(p) r^\sigma_\alpha(q).
\end{equation}
If $\sigma_x = \sigma_y$ for all $x,y$, then $\alpha \mapsto r^\sigma_\alpha$ is a
continuous $\Per(\Lambda)\widehat{\;}$-valued $1$-cocycle on $\Gg_\Lambda$.
\end{lem}
\begin{proof}
The map $\alpha \mapsto r^\sigma_\alpha(p)$ is continuous because $\sigma$ is continuous.

A straightforward calculation in the central extension $\Gg_\Lambda \times_{\sigma} \TT$
shows that for $w,z \in \TT$, we have
\[
(\alpha, w) ((y,p,y), z)(\alpha, w)^{-1} = ((x,p,x), r^\sigma_\alpha(p)z).
\]

Computing further in the central extension, we have
\begin{align*}
((x,p+q,x), &r^\sigma_\alpha(p+q)) \\
	&= (\alpha, 1) ((y,p+q,y), 1)(\alpha, 1)^{-1}\\
	&= (\alpha, 1) ((y,p,y), \overline{\sigma_y(p,q)})(\alpha, 1)^{-1}	(\alpha, 1) ((y,q,y), 1)(\alpha, 1)^{-1}\\
	&= ((x,p,x), r^\sigma_\alpha(p)\overline{\sigma_y(p,q)})
		((x,q,x), r^\sigma_\alpha(q))\\
	&= \big((x,p+q,x), \sigma_x(p,q)\overline{\sigma_y(p,q)}
		r^\sigma_\alpha(p) r^\sigma_\alpha(q)\big),
\end{align*}
giving~\eqref{eq:rsigmaformula}.

Now suppose that $\sigma_x = \sigma_y$ for all $x,y$. Then~\eqref{eq:rsigmaformula}
implies immediately that $r^\sigma_\alpha \in \Per(\Lambda)\widehat{\;}$ for all
$\alpha$. Take $p \in \Per(\Lambda)$ and composable $\alpha,\beta \in \Gg_\Lambda$. Let
$y = s(\beta)$. Computing again in $\Gg_\Lambda \times_\sigma \TT$, we have
\begin{align*}
\big((r(\alpha), p, r(\alpha)), r^\sigma_\alpha(p)r^\sigma_\beta(p)\big)
    &= (\alpha, 1)(\beta, 1) \big((y,p,y), 1\big) (\beta, 1)^{-1} (\alpha,1)^{-1}\\
    &= (\alpha\beta, \sigma(\alpha,\beta)) \big((y,p,y), 1\big)
        (\alpha\beta, \sigma(\alpha,\beta))^{-1}\\
    &= \big((r(\alpha), p, r(\alpha)), r^\sigma_{\alpha\beta}(p)\big).
\end{align*}
So $\alpha \mapsto r^\sigma_\alpha$ is a $1$-cocycle on $\Gg_\Lambda$.
\end{proof}

Given a cocycle $\omega \in Z^2(\Per(\Lambda), \TT)$, we write $\omega^*$ for the cocycle
$(p,q) \mapsto \overline{\omega(q,p)}$. By \cite[Proposition~3.2]{OPT} (see also \cite[Lemma 7.1]{Kleppner}),
the map $\omega\omega^*$ is a bicharacter of $\Per(\Lambda)$ which is antisymmetric\footnote{In
\cite{OPT}, the word ``symplectic" is used instead of antisymmetric.} in the sense that
$(\omega\omega^*)(p,q) = \overline{(\omega\omega^*)(q,p)}$. Proposition~3.2 of \cite{OPT}
implies that $\omega \mapsto \omega\omega^*$ descends to an isomorphism of
$H^2(\Per(\Lambda), \TT)$ onto the group $X^2(\Per(\Lambda), \TT)$ of all antisymmetric
bicharacters of $\Per(\Lambda)$.

\begin{lem}\label{lem:bicharacter}
Let $\Lambda$ be a cofinal row-finite $k$-graph with no sources, and suppose $c \in
\Zcat2(\Lambda, \TT)$. Let $\sigma_c$ be a continuous cocycle on $\Gg_\Lambda$ of the
form~\eqref{eq:sigmac},  so that $C^*(\Gg_\Lambda, \sigma_c) \cong C^*(\Lambda, c)$. For
each $x \in \Lambda^\infty$, let $\sigma^x_c \in Z^2(\Per(\Lambda), \TT)$ be the cocycle
given by $\sigma^x_c(p,q) = \sigma_c((x,p,x),(x,q,x))$. Then the cohomology class of
$\sigma^x_c$ is independent of $x$.
\end{lem}
\begin{proof}
The formula~\eqref{eq:sigmac} shows that $\sigma_c$ is locally constant as a function
from $\Gg_\Lambda \times \Gg_\Lambda \to \TT$. Restricting $\sigma_c$ to $\Ii_\Lambda$,
we obtain cocycles on the groups $(\Ii_\Lambda)_x = \{(x,p,x) : p \in \Per(\Lambda)\}
\cong \Per(\Lambda)$, and hence cocycles $\sigma_c^x$ on $\Per(\Lambda)$ as claimed. For
$x \in \Lambda^\infty$, let $\omega_x$ be the bicharacter of $\Per(\Lambda)$ given by
\[
\omega_x(p,q) := \sigma_c\big((x,p,x),(x,q,x)\big) \overline{\sigma_c\big((x,q,x),(x,p,x)\big)}.
\]
Fix free abelian generators $g_1, \dots, g_l$ for $\Per(\Lambda)$. Since each $\omega_x$
is a bicharacter, it is determined by the values $\omega_x(g_i, g_j)$. Fix $x \in
\Lambda^\infty$. Since $\sigma_c$ is locally constant, for each $i,j$ there is a
neighbourhood $U_{i,j}$ of $x$ such that $\sigma^x_c(g_i, g_j) = \sigma_c((x, g_i, x),
(x, g_j, x))$ is constant on $U_{i,j}$. So for $y \in U := \bigcap_{i,j} U_{i,j}$ we have
$\omega_y(g_i, g_j) = \omega_x(g_i, g_j)$ for all $i,j$, and hence $\omega_y = \omega_x$.
Now \cite[Proposition~3.2]{OPT} implies that the cohomology class (in $H^2(\Per(\Lambda),
\TT)$) of $\sigma_c^x$ is locally constant with respect to $x$. Since $\Lambda$ is
cofinal, $\Gg_\Lambda$ is minimal, and so every orbit in $\Gg_\Lambda$ is dense; so to
see that the cohomology class of $\sigma_c^x$ is globally constant, it suffices to show
that it is constant on orbits. For $x \in \Lambda^\infty$, let $A_x$ denote the subgroup
$\{((x,p,x), z) : p \in \Per(\Lambda), z \in \TT\} \subseteq (\Gg_\Lambda
\times_{\sigma_c} \TT)^x_x$, which is isomorphic to the group extension $\Per(\Lambda)
\times_{\sigma^x_c} \TT$ of $\Per(\Lambda)$ by $\TT$. Conjugation by any $\alpha$ in
$\Gg_\Lambda \times_{\sigma_c} \TT$ is an isomorphism $\Ad_\alpha : A_{s(\alpha)} \cong
A_{r(\alpha)}$. For $\gamma \in \Gg_\Lambda$, let $r^{\sigma_c}_\gamma : \Per(\Lambda)
\to \TT$ be the map of Lemma~\ref{lem:r-alpha}. Fix $\alpha = (\gamma, w) \in \Gg_\Lambda
\times_{\sigma_c} \TT$, $p\in \Per(\Lambda)$ and $z \in \TT$, and let $x = r(\gamma)$ and
$y = s(\gamma)$. A quick calculation gives $\Ad_\alpha((y,p,y), z) = ((x,p,x),
r^{\sigma_c}_{\gamma}(p)z)$. If $p = 0$, then~\eqref{eq:conjugation} collapses to give
$\Ad_\alpha((y,0,y),z) = ((x,0,x),z)$ because $\sigma_c((x,m,y),(y,0,y)) = 1$. If $q :
\Ii_\Lambda \times_{\sigma_c} \TT \to \Lambda^\infty \times \Per(\Lambda)$ is the
quotient map $((x,p,x), z) \mapsto (x,p)$, then
\[
q(\Ad_\alpha((y,p,y), z))
	= q\big((x,p,x), r^{\sigma_c}_{\gamma}(p)z\big)
	= (x,p),
\]
so that $\Ad_\alpha$ descends through $q$ to the map $(y,p) \mapsto (x,p)$. Thus
conjugation by $\alpha$ determines an isomorphism
\[\begin{tikzpicture}[>=stealth]
	\node (0tl) at (0,1) {$0$};
	\node (Tt) at (2,1) {$\TT$};
	\node (At) at (4,1) {$A_y$};
	\node (Pt) at (6,1) {$\Per(\Lambda)$};
	\node (0tr) at (8,1) {$0$};
	\draw[->] (0tl)--(Tt);
	\draw[->] (Tt)--(At);
	\draw[->] (At)--(Pt);
	\draw[->] (Pt)--(0tr);
	\node (0bl) at (0,-1) {$0$};
	\node (Tb) at (2,-1) {$\TT$};
	\node (Ab) at (4,-1) {$A_x$};
	\node (Pb) at (6,-1) {$\Per(\Lambda)$};
	\node (0br) at (8,-1) {$0$};
	\draw[->] (0bl)--(Tb);
	\draw[->] (Tb)--(Ab);
	\draw[->] (Ab)--(Pb);
	\draw[->] (Pb)--(0br);
	\draw[->] (Tt)--(Tb) node[pos=0.5, left] {\small$\id$};
	\draw[->] (At)--(Ab) node[pos=0.5, left] {\small$\Ad_\alpha$};
	\draw[->] (Pt)--(Pb) node[pos=0.5, left] {\small$\id$};
\end{tikzpicture}\]
of extensions.
Hence $\sigma^x_c$ is cohomologous to $\sigma^y_c$; so the cohomology class of $\sigma_c^x$ is constant on orbits.
\end{proof}

\begin{proof}[Proof of Proposition~\ref{prp:constant in Ii}]
Let $\sigma_c$ be a continuous cocycle on $\Gg_\Lambda$ constructed in
Section~\ref{sec:kgrgpds} so that $C^*(\Gg_\Lambda, \sigma_c) \cong C^*(\Lambda, c)$. By
Lemma~\ref{lem:bicharacter}, the cohomology class of $\sigma_c^x$ is independent of $x$.
So there exists a cocycle $\omega \in Z^2 (\Per(\Lambda), \TT)$ whose cohomology class
agrees with that of $\sigma_c^x$ for each $x$. We may assume that $\omega$ is a
bicharacter.

The map $\tilde{c}_x : \Per(\Lambda) \times \Per(\Lambda) \to \TT$ defined by
\[
(p,q) \mapsto \tilde{c}_x (p,q) := \sigma_c((x,p,x),(x,q,x)) \overline{\omega(p,q)}
\]
is a coboundary on $\Per(\Lambda)$ for each $x$. Fix free abelian generators $g_1, \dots,
g_l$ for $\Per(\Lambda)$. Fix $x \in \Lambda^\infty$, and define $b_x(0) = b_x(g_j) = 1
\in \TT$ for all $j$, and then define $b_x$ on $\Per(\Lambda)$ inductively by
\begin{equation}\label{eq:bxdef}
b_x(m)\overline{b_x(m + g_i)} = \tilde{c}_x(g_i, m)
	\qquad\text{whenever $m \in \langle g_j : j \le i\rangle$.}
\end{equation}
Since $x \mapsto \tilde{c}_x(p,q)$ is continuous for each $p,q$, an induction argument
shows that $x \mapsto b_x(p)$ is continuous for each $p$.

We claim that each $\delta^1 b_x = \tilde{c}_x$. To see this, choose for each $x$ a
cochain $\tilde{b}_x : \Per(\Lambda) \to \TT$ such that $\delta^1 \tilde{b}_x =
\tilde{c}_x$. The map $a_x(m) := \prod_{i=1}^l \overline{\tilde{b}_x(g_i)}^{m_i}$ is a
1-cocycle on $\Per(\Lambda)$, and so $\delta^1 a_x = 1$. Hence $\delta^1 (a_x
\tilde{b}_x) = \delta^1 \tilde{b}_x = \tilde{c}_x$. We have $(a_x \tilde{b}_x)(0) = (a_x
\tilde{b}_x)(g_i) = 1$ for all $i$, and for $i \le l$ and $m \in \langle g_j : j \le
i\rangle$, we have
\begin{align*}
(a_x\tilde{b}_x)(m)\overline{(a_x\tilde{b}_x)(m+g_i)}
	&= (a_x\tilde{b}_x)(g_i)(a_x\tilde{b}_x)(m)
		\overline{(a_x\tilde{b}_x)(m+g_i)} \\
	&= \delta^1(a_x\tilde{b}_x)(g_i, m)
	= \tilde{c}_x (g_i,m).
\end{align*}
So $b_x$ and $a_x \tilde{b}_x$ both take $0$ and each $g_i$ to $1$ and
satisfy~\eqref{eq:bxdef}. Hence $b_x = a_x\tilde{b}_x$, and $\delta^1b_x = \tilde{c}_x$
for all $x$.

Since $x \mapsto b_x(p)$ is continuous for each $p$, the map $b : (x,p,x) \mapsto b_x(p)$
is a continuous cochain on $\Ii_\Lambda$. Since $\Ii_\Lambda$ is clopen in $\Gg_\Lambda$,
we can extend $b$ to a cochain $\tilde{b}$ on all of $\Gg_\Lambda$ by setting
$\tilde{b}|_{\Gg_\Lambda \setminus \Ii_\Lambda} \equiv 1$.

Now $\delta^1\tilde{b}$ is a continuous coboundary on $\Gg_\Lambda$, so $\sigma :=
\sigma_c \cdot \overline{\delta^1 \tilde{b}}$ represents the same cohomology class as
$\sigma_c$. Hence $C^*(\Gg_\Lambda, \sigma) \cong C^*(\Gg_\Lambda, \sigma_c) \cong
C^*(\Lambda, c)$ \cite[Proposition~II.1.2]{Ren}. We have $\sigma((x,p,x),(x,q,x)) =
\omega(p,q)$ for $p,q \in \Per(\Lambda)$ by construction of $\tilde{b}$. Each
$\sigma^x_c$ is cohomologous to $\omega$ by choice of $\omega$.
\end{proof}

Given an abelian group $A$ and a cocycle $\omega \in Z^2(A, \TT)$, we write
\begin{equation}\label{eq:Zomega}
    Z_\omega := \{p \in A : (\omega\omega^*)(p,q) = 1 \text{ for all } q \in A\}
\end{equation}
for the kernel of the homomorphism $p \mapsto (\omega\omega^*)(p, \cdot)$ from $A$ to
$\widehat{A}$ induced by $\omega\omega^*$. Thus $Z_\omega$ is a subgroup of $A$.

We now have all the ingredients needed to state our main theorem.

\begin{thm}\label{thm:main simple}
Let $\Lambda$ be a row-finite cofinal $k$-graph with no sources, and suppose that $c \in
\Zcat2(\Lambda, \TT)$. Suppose that $\sigma \in Z^2(\Gg_\Lambda, \TT)$ and $\omega \in
Z^2(\Per(\Lambda), \TT)$ satisfy $C^*(\Gg_\Lambda, \sigma) \cong C^*(\Lambda, c)$ and
$\sigma|_{\Ii_\Lambda} = \omega \times 1_{\Lambda^\infty}$ (such a pair $\sigma, \omega$
exist by Proposition~\ref{prp:constant in Ii}). Let $r^\sigma : \Gg_\Lambda \to
\Per(\Lambda)\widehat{\;}$ be the cocycle of Lemma~\ref{lem:r-alpha}. Then $C^*(\Lambda,
c)$ is simple if and only if $\{(r(\gamma), r^\sigma_\gamma|_{Z_\omega}) : \gamma \in
(\Gg_\Lambda)_x\}$ is dense in $\Lambda^\infty \times \widehat{Z}_\omega$ for all $x \in
\Lambda^\infty$. In particular, if $\omega$ is nondegenerate, then $C^*(\Lambda, c)$ is
simple.
\end{thm}

The proof of the main theorem will occupy the remainder of this section and most of the
next. Before we get started, we provide a practical method for computing $Z_\omega$
without reference to $\Gg_\Lambda$. To see why this is useful, observe that to apply our
main theorem, it is typically necessary to compute a cocycle $\sigma$ on $\Gg_\Lambda$
with the required properties, and this is not so easy to do. (We discuss a class of
examples where this is possible in Section~\ref{sec:CP}.) But the last statement of the
theorem says that if we know that the centre of the bicharacter $\omega$ is trivial, then
no computations in $\Gg_\Lambda$ are necessary.

In the following result, for $m \in \ZZ^k$, we write $m^+$ and $m^-$ for $m \vee 0$ and
$(-m) \vee 0)$. So $m = m^+ - m^-$ and $m^+ \wedge m^- = 0$.

\begin{lem}\label{lem:compute omega}
Let $\Lambda$ be a row-finite cofinal $k$-graph with no sources, and suppose that $c \in
\Zcat2(\Lambda, \TT)$. Let $g_1, \dots, g_l$ be free generators for $\Per(\Lambda)$.
There exists $v \in \Lambda^0$ such that $\shift^{g_i^+}(x) = \shift^{g_i^-}(x)$ for all
$x \in Z(v)$ and $1 \le i \le l$. Let $N = \sum^l_{i=1} g^+_i + g^-_i$, and fix $\lambda \in v\Lambda^N$. For
$1 \le i \le l$, factorise $\lambda = \mu_i \tau_i = \nu_i\rho_i$ where $d(\mu_i) = g^+_i$ and
$d(\nu_i) = g^-_i$. For $1 \le i, j \le l$, factorise $\lambda = \mu_{ij}\tau_{ij} =
\nu_{ij}\rho_{ij}$ where $d(\mu_{ij}) = (g_i + g_j)^+$ and $d(\nu_{ij}) = (g_i + g_j)^-$.
Let $\omega$ be the bicharacter of $\Per(\Lambda)$ such that
\[
\omega(g_i, g_j) := c(\mu_i,\tau_i)\overline{c(\nu_i,\rho_i)}c(\mu_j, \tau_j)\overline{c(\nu_j, \rho_j)}
    \overline{c(\mu_{ij}, \tau_{ij})}c(\nu_{ij}, \rho_{ij})
    \quad\text{ for $1 \le i,j \le l$}.
\]
Then there exist $\sigma \in Z^2(\Gg_\Lambda, \TT)$ and a bicharacter $\omega'$ of
$\Per(\Lambda)$ such that $C^*(\Lambda,c) \cong C^*(\Gg_\Lambda, \sigma)$,
$\sigma|_{\Ii_\Lambda} = \omega' \times  1_{\Lambda^\infty}$, and $Z_{\omega'} =
Z_\omega$. In particular, if $Z_\omega = \{0\}$ then $C^*(\Lambda, c)$ is simple.
\end{lem}
\begin{proof}
We claim that there is a set $\Pp \subseteq \Lambda \mathbin{{_s*_s}} \Lambda$ such that
\[
\{(\lambda, s(\lambda)) : \lambda \in \Lambda\} \cup \{(\mu_i, \nu_i) : i \le l\} \cup \{(\mu_{ij}, \nu_{ij}) : i \not= j\}
    \subseteq \Pp,
\]
and such that $\Gg_\Lambda = \bigsqcup_{(\mu,\nu) \in\Pp} Z(\mu,\nu)$.

To see this, we argue as in \cite[Lemma~6.6]{KPS4} to see that the $Z(\lambda,
s(\lambda))$ are mutually disjoint and $\bigsqcup_\lambda Z(\lambda, s(\lambda))$ is
clopen in $\Gg_\Lambda$. Let
\[
\Pp_0 := \{(\lambda, s(\lambda)) : \lambda \in \Lambda\} \cup \{(\mu_i, \nu_i) : i \le l\}
    \cup \{(\mu_{ij}, \nu_{ij}) : 1 \le i \le j \le l\}.
\]
For each $i \le l$ either $(\mu_i, \nu_i) = (\mu_i, s(\mu_i))$, or $d(\mu_i) - d(\nu_i)
\not\in \NN^k$ and so $Z(\mu_i, \nu_i) \cap Z(\lambda, s(\lambda)) = \emptyset$ for all
$\lambda$; and likewise for each $(\mu_{ij}, \nu_{ij})$. Each $Z(\mu_i, \nu_i) \subseteq
\Lambda^\infty \times \{g_i\} \times \Lambda^\infty$, and each $Z(\mu_{ij}, \nu_{ij})
\subseteq \Lambda^\infty \times \{g_i + g_j\} \times \Lambda^\infty$. So the $Z(\mu_i,
\nu_i)$ and the $Z(\mu_{ij},\nu_{ij})$ collectively are mutually disjoint. Hence the
$Z(\mu,\nu)$ where $(\mu,\nu) \in \Pp_0$ are mutually disjoint, and $\bigsqcup_{(\mu,\nu)
\in \Pp_0} Z(\mu,\nu)$ is clopen in $\Gg_\Lambda$. Now the argument of the final two
paragraphs of \cite[Lemma~6.6]{KPS4} shows that we can extend $\Pp_0$ to the required
collection $\Pp$.

The formula~\eqref{eq:sigmac} yields a cocycle $\sigma_c \in Z^2(\Gg_\Lambda, \TT)$, and
the construction of $\Pp_0$ shows that
\[
\sigma_c((x, g_i, x), (x, g_j, x)) = \omega(g_i, g_j)
    \quad\text{ for all $x \in Z(\lambda)$ and $i,j \le l$.}
\]
Corollary~7.9 of \cite{KPS4} implies that $C^*(\Lambda, c) \cong C^*(\Gg_\Lambda,
\sigma_c)$. Proposition~3.1 applied to this $\Pp$ gives a bicharacter $\omega'$ of
$\Per(\Lambda)$ and a cocycle $\sigma$ on $\Gg_\Lambda$ such that $C^*(\Gg_\Lambda,
\sigma) \cong C^*(\Gg_\Lambda, \sigma_c) \cong C^*(\Lambda, c)$ and $\omega'$ is
cohomologous to $\sigma^x_c$ and hence to $\omega$. Thus $Z_\omega = Z_{\omega'}$ by
\cite[Proposition~3.2]{OPT}. The final statement follows from Theorem~\ref{thm:main
simple}.
\end{proof}

We finish the section by showing that $r^\sigma$ induces an action $\theta$ of the
quotient $\Hh_\Lambda$ of $\Gg_\Lambda$ by the interior of its isotropy on the space
$\Lambda^\infty \times \widehat{Z}_\omega$. In particular, we prove that $\theta$ is
minimal if and only if $\{(r(\gamma), r^\sigma_\gamma|_{Z_\omega}) : \gamma \in
(\Gg_\Lambda)_x\}$ is dense in $\Lambda^\infty \times \widehat{Z}_\omega$ for all $x \in
\Lambda^\infty$ as in Theorem~\ref{thm:main simple}. In the next section we will realise
$C^*(\Lambda, c)$ as the $C^*$-algebra of a Fell bundle over a quotient groupoid
$\Hh_\Lambda$ of $\Gg_\Lambda$, and show that the action, given by \cite{IW}, of
$\Hh_\Lambda$ on the primitive ideal space of the $C^*$-algebra over the unit space in
this bundle is isomorphic as a groupoid action to $\theta$. We then prove our main
theorem by showing that minimality of the action described in \cite{IW} characterises
simplicity of the $C^*$-algebra of the Fell bundle.

Given $\Lambda, c$ and $\omega$ as in Proposition~\ref{prp:constant in Ii}, and with
$Z_\omega$ as in~\eqref{eq:Zomega}, we can form the quotient $H :=
\Per(\Lambda)/Z_\omega$. We then have $\widehat{H} \cong Z_\omega^\perp \leq
\Per(\Lambda)\widehat{\;}$, so we regard $\widehat{H}$ as a subgroup of
$\Per(\Lambda)\widehat{\;}$.

\begin{lem}\label{lem:tilde-r cocycle}
Let $\Lambda$ be a row-finite cofinal $k$-graph with no sources, and suppose that $c \in
\Zcat2(\Lambda, \TT)$. Suppose that $\sigma \in Z^2(\Gg_\Lambda, \TT)$ is cohomologous to
$\sigma_c$ and that $\omega \in Z^2(\Per(\Lambda), \TT)$ satisfies $\sigma|_{\Ii_\Lambda}
= 1_{\Lambda^\infty} \times \omega$ as in Proposition~\ref{prp:constant in Ii}. Let
$r^\sigma$ be the $\Per(\Lambda)\widehat{\;}$-valued $1$-cocycle on $\Gg_\Lambda$ given
by Lemma~\ref{lem:r-alpha}. For $\alpha \in \Ii_\Lambda$, we have $r^\sigma_\alpha \in
Z_\omega^\perp$. Let $\pi : \Gg_\Lambda \to \Hh_\Lambda := \Gg_\Lambda/\Ii_\Lambda$ be
the quotient map. There is a continuous $\widehat{Z}_\omega$-valued $1$-cocycle
$\tilde{r}^\sigma$ on $\Hh_\Lambda$ such that $\tilde{r}^\sigma_{\pi(\alpha)}(p) =
r^\sigma_\alpha(p)$ for all $\alpha \in \Gg_\Lambda$ and $p \in Z_\omega$. There is an
action $\theta$ of $\Hh_\Lambda$ on $\Lambda^\infty \times \widehat{Z}_\omega$ such that
\[
\theta_\alpha(s(\alpha), \chi) = (r(\alpha), \tilde{r}^\sigma_\alpha \cdot \chi)
    \text{ for all $\alpha \in \Hh_\Lambda$ and $\chi \in \widehat{Z}_\omega$.}
\]
\end{lem}
\begin{proof}
Suppose that $p \in \Per(\Lambda)$ and $q \in Z_\omega$. Calculating in $\Gg_\Lambda \times_\sigma \TT$, we have
\begin{align*}
((x,p,x),1)&((x,q,x), 1)((x,p,x),1)^{-1} \\
	&= \big((x,q,x), \sigma\big((x,p,x),(x,q,x)\big)
					\overline{\sigma\big((x,q,x),(x,p,x)\big)}\big) \\
    &= \big((x,q,x), \omega\omega^*(p,q)\big).
\end{align*}
Hence $r^\sigma_{(x,p,x)}(q) = (\omega\omega^*)(p,q) = 1$ since $q \in Z_\omega$.

Suppose that $\pi(\alpha) = \pi(\beta)$. Then $\alpha = \beta \gamma$ for some $\gamma
\in \Ii_\Lambda$. So for $p \in Z_\omega$ we have $r^\sigma_\alpha(p) =
r^\sigma_\beta(p)r^\sigma_\gamma(p) = r^\sigma_\beta(p)$, showing that $\tilde{r}^\sigma$
is well defined. It is continuous by definition of the quotient topology, and is a
1-cocycle because $r^\sigma$ is. For $p, p' \in Z_\omega$, we have
\[
\tilde{r}^\sigma_{\pi(\alpha)}(p + p')
	= r^\sigma_\alpha(p + p')
	= r^\sigma_\alpha(p) r^\sigma_\alpha(p')
	= \tilde{r}^\sigma_{\pi(\alpha)}(p)\tilde{r}^\sigma_{\pi(\alpha)}(p').
\]
The final statement follows immediately.
\end{proof}

\section{A Fell bundle associated to a \texorpdfstring{$k$}{k}-graph 2-cocycle}\label{sec:Fell bundle}

To identify the twisted $k$-graph algebra $C^*(\Lambda, c)$ with the $C^*$-algebra of a
Fell bundle, we must first construct the bundle. We start by describing what will become
the $C^*$-algebra sitting over the unit-space in this bundle: the twisted $C^*$-algebra
of $\Ii_\Lambda$ sitting inside that of $\Gg_\Lambda$.

\begin{lem}\label{lem:twisted Per}
Let $\Lambda$ be a row-finite cofinal $k$-graph with no sources, and take $c \in
\Zcat2(\Lambda, \TT)$. Suppose that $\sigma \in Z^2(\Gg_\Lambda, \TT)$ is cohomologous to
$\sigma_c$ and that $\omega \in Z^2(\Per(\Lambda), \TT)$ satisfies $\sigma|_{\Ii_\Lambda}
= 1_{\Lambda^\infty} \times \omega$ as in Proposition~\ref{prp:constant in Ii}. There is
a $*$-homomorphism
\[
\Phi : C_c(\Lambda^\infty) \otimes C^*(\Per(\Lambda), \omega) \to C_c(\Ii_\Lambda, \sigma)
\]
such that $(f \otimes u_p)(x,q,x) = \delta_{p,q}f(x)$ for all $f \in C_c(\Lambda^\infty)$
and $p,q \in \Per(\Lambda)$. This $*$-homomorphism extends to an isomorphism $\Phi :
C_0(\Lambda^\infty) \otimes C^*(\Per(\Lambda, \omega)) \to C^*(\Ii_\Lambda, \sigma)$.
\end{lem}
\begin{proof}
Recall that we identify $\Ii_\Lambda$ with $\Lambda^\infty \times \Per(\Lambda)$ as in
Corollary~\ref{cor:ILambda}. For each $p \in \Per(\Lambda)$, the characteristic function
$U_p := 1_{\Lambda^\infty \times \{p\}}$ is a unitary multiplier of $C^*(\Ii_\Lambda,
\sigma)$. By construction of $\sigma$, we have $U_p U_q = \omega(p,q) U_{p+q}$ for all
$p,q$, so the universal property of $C^*(\Per(\Lambda), \omega)$ gives a homomorphism
$C^*(\Per(\Lambda), \omega) \to \Mm C^*(\Ii_\Lambda, \sigma)$ satisfying $u_p \mapsto
U_p$. The $U_p$ clearly commute with $C_0(\Lambda^\infty)$, so the universal property of
the tensor product gives a homomorphism $C_0(\Lambda^\infty) \otimes C^*(\Per(\Lambda),
\omega) \to C^*(\Ii_\Lambda, \sigma)$ satisfying $f \otimes u_p \mapsto U_p f$. For $x
\in \Lambda^\infty$, we have $(\Ii_\Lambda)_x \cong \Per(\Lambda)$, and
$\sigma|_{(\Ii_\Lambda)_x} = \omega$ by Proposition~\ref{prp:constant in Ii}. Hence the
regular representation of $C^*(\Ii_\Lambda, \sigma)$ on $\ell^2((\Ii_\Lambda)_x)$ is
unitarily equivalent to the canonical faithful representation of $C^*(\Per(\Lambda),
\omega)$ on $\ell^2(\Per(\Lambda))$. Thus the homomorphism $\rho : f \otimes u_p \mapsto
U_p f$ is a fibrewise-injective homomorphism of $C_0(\Lambda^\infty)$-algebras, and is
injective on the central copy of $C_0(\Lambda^\infty)$. Since the norm on a
$C_0(X)$-algebra is the same as the supremum norm on the algebra of sections of the
associated upper-semicontinuous bundle \cite{Williams:crossed}, it follows that $\rho$ is
injective. Since its range contains $C_0(\Lambda^\infty \times \{p\})$ for each $p$, it
is also surjective.
\end{proof}

Given a compact abelian group $G$, and an action $\beta$ of a closed subgroup $H$ of $G$
on a $C^*$-algebra $C$, the \emph{induced algebra} $\Ind^G_H C$ is defined by
\[
\Ind^G_H C
    = \{f : G \to C \mid f(g - h) = \beta_h(f(g))\text{ for } g \in G\text{ and }h \in H\}
\]
with pointwise operations. Note that we use additive notation to emphasize that $G$ must be abelian. This algebra carries an induced action $\lt$ of $G$ given by
translation: $\lt_g(f)(g') = f(g'-g)$. For $h \in H$, we have $\lt_h(f)(g)=
\beta_h(f(g))$.

If $C$ is simple, then $\Prim(\Ind^G_H C)$ is homeomorphic to $G/H$:  for $g + H \in
G/H$, the corresponding primitive ideal $I_{g+H}$ is the ideal $\{f \in \Ind^G_H C :
f|_{g+H} = 0\}$ (see, for example, \cite[Proposition~6.6]{tfb}).

Now let $A$ be a discrete abelian group and let $\omega$ be a $\TT$-valued 2-cocycle on
$A$. Let $Z_\omega \subseteq A$ be the centre of $\omega$ as in~\eqref{eq:Zomega}. The
antisymmetric bicharacter $\omega\omega^*$ descends to an antisymmetric bicharacter
$(\omega\omega^*)\tilde{}$ of $B := A/Z_{\omega}$. There is a cocycle $\tilde{\omega} \in
Z^2(B, \TT)$ such that
\[
\tilde{\omega}\tilde{\omega}^*(a + Z_\omega, a' + Z_\omega) = (\omega\omega^*)(a, a')
\text{ for all $a, a' \in A$}.
\]
By construction, the antisymmetric bicharacter $(\omega\omega^*)\tilde{}$ is
nondegenerate in the sense that $g + Z_\omega \mapsto (\omega\omega^*)\tilde{}(g +
Z_\omega, \cdot)$ is injective (as a map from $B$ to its dual $\widehat{B}$).

There is an action $\beta^A : \widehat{A} \to \Aut(C^*(A, \omega))$ such that
$\beta^A_t(U_a) = \chi(a) U_a$ for $\chi \in \widehat{A}$. Recall from \cite[Lemma~5.11
and Theorem~6.3]{OPT}\footnote{There is a typographical error in the statement of
\cite[Theorem~6.3]{OPT}: $G/G_{\mathscr{Z}}$ should read $G_{\mathscr{Z}}$.} that there
is an $\widehat{A}$-equivariant isomorphism
\[
C^*(A, \omega) \cong
    \Ind^{\widehat{A}}_{\widehat{B}} C^*(B, \tilde{\omega}),
\]
and that $C^*(B, \tilde{\omega})$ is simple. Hence $\Prim(C^*(A, \omega)) \cong
\widehat{A}/\widehat{B} \cong \widehat{Z_{\omega}}$ \cite[Proposition~6.6]{tfb}. In
particular, in the situation of Lemma~\ref{lem:twisted Per}, we have
\begin{align*}
\Prim(C^*(\Ii_\Lambda, \sigma)) &\cong
\Prim\big(C_0(\Lambda^\infty) \otimes C^*(\Per(\Lambda), \omega) \big) \cong
\Lambda^\infty \times \widehat{Z_{\omega}}.
\end{align*}

Now resume the hypotheses and notation of Lemma~\ref{lem:twisted Per}. We construct a
Fell bundle over $\Hh_\Lambda$. We describe the fibres of the bundle and the
multiplication and involution operations first, and then prove in
Proposition~\ref{prp:bundle} that there is a topology compatible with these operations
under which we obtain the desired Fell bundle.  We write $C^*(\Hh; \sB)$ for the full
$C^*$-algebra of a Fell bundle $\sB$ over a groupoid $\Hh$, and $C^*_r(\Hh; \sB)$ for the
corresponding reduced $C^*$-algebra. We make the convention that $C^*(\Hh^{(0)}; \sB)$
denotes the full $C^*$-algebra of the restriction of $\sB$ to the unit space; this
$C^*(\Hh^{(0)}; \sB)$ is a $C_0(\Hh^{(0)})$-algebra. For background on Fell bundles over
\'etale groupoids, see \cite{Kumjian98, MW}.

We identify both $\Hh_\Lambda^0$ and $\Gg_\Lambda^{(0)}$ with $\Lambda^\infty$.  We
continue to write $\pi : \Gg_\Lambda \to \Hh_\Lambda$ for the quotient map, and we often
write $[\gamma]$ for $\pi(\gamma)$, and regard it as an equivalence class in
$\Gg_\Lambda$; that is $[\gamma] = \{\alpha \in \Gg_\Lambda : \pi(\alpha) = \pi(\gamma)\}
= \gamma \cdot \Ii_\Lambda$.

We define $A_x = C^*(\Per(\Lambda), \omega)$ for $x \in \Lambda^\infty$, and we write
$\sA_\Lambda$ for the trivial bundle $\Lambda^\infty \times C^*(\Per(\Lambda), \omega)$.
So $C^*(\Ii_\Lambda, \sigma) \cong C^*(\Lambda^\infty; \sA_\Lambda)$, and $A_x$ is the
fibre of $\sA_\Lambda$ over $x$. For $[\gamma] \in \Hh_\Lambda$, we let
$B^\circ_{[\gamma]} := C_c([\gamma])$ as a vector space over $\CC$, and we define
multiplication $B^\circ_{[\alpha]} \times B^\circ_{[\beta]} \to B^\circ_{[\alpha\beta]}$
where $s(\alpha) = r(\beta)$, and involution $B^\circ_{[\alpha]} \to
B^\circ_{[\alpha^{-1}]}$ by
\[
(f * g)(\gamma) = \sum_{\substack{\eta\zeta = \gamma,\\ \eta \in [\alpha], \zeta \in [\beta]}}
\sigma(\eta,\zeta)f(\eta)g(\zeta)
\qquad\text{ and }\qquad
f^*(\gamma) = \overline{\sigma(\gamma, \gamma^{-1})} \overline{f(\gamma^{-1})}.
\]

We regard each $B^\circ_{[x]}$ (where $x \in \Lambda^\infty = \Hh_\Lambda^{(0)}$) as a
dense subspace of $A_x$, and endow it with the norm inherited from $A_x$. We write $B_x
:= \overline{B^\circ_{[x]}} \cong A_x$, and identify $B_x$ and $A_x$. For $[\gamma] \in
\Hh_\Lambda$, we define an $A_{s(\gamma)}$-valued inner product on $B^\circ_{[\gamma]}$
by $\langle f, g\rangle_{s(\gamma)} := f^* * g$. Then we obtain a norm on
$B^\circ_{[\gamma]}$ by $\|f\| = \|\langle f, f\rangle_{s(\gamma)}\|^{1/2}$, and we write
$B_{[\gamma]}$ for the completion of $B^\circ_{[\gamma]}$ in this norm. Note that
$B_{[\gamma]}$ is a full right $A_{s(\gamma)}$-Hilbert module. It is straightforward to
show that $\|f * g\| \le \| f\|\|g\|$ for all $f \in B^\circ_{[\alpha]}$ and $g \in
B^\circ_{[\beta]}$ when $s(\alpha) = r(\beta)$, and involution extends to an isometric
conjugate linear map $B_{[\alpha]} \to B_{[\alpha^{-1}]}$.

\begin{prop}\label{prp:bundle}
With notation as above there is a unique topology on $\sB_\Lambda = \bigsqcup_{[\gamma]
\in \Hh_\Lambda} B_{[\gamma]}$ under which it is a Banach bundle and such that for each
$f \in C_c(\Gg_\Lambda, \sigma)$, the section $[\gamma] \mapsto f|_{[\gamma]}$ is norm
continuous. Under this topology, the space $\sB_\Lambda$ is a saturated continuous Fell
bundle over $\Hh_\Lambda$.
\end{prop}
\begin{proof}
The first assertion follows from \cite[Proposition~10.4]{Fell77} once we show that the
sections $[\gamma] \mapsto f|_{[\gamma]}$ associated to elements $f \in C_c(\Gg_\Lambda,
\sigma)$ are norm continuous and the ranges of these sections are pointwise dense.  Let
$f \in C_c(\Gg_\Lambda, \sigma)$. Then the restriction $[x] \mapsto \big\|f|_{[x]}\big\|$
is continuous on $\Hh_\Lambda^{(0)} = \Lambda^\infty$ because $C^*(\Ii_\Lambda, \sigma)
\cong C_0(\Lambda^\infty) \otimes C^*(\Per(\Lambda), \omega)$ by Lemma~\ref{lem:twisted
Per}. To show  $[\gamma] \mapsto \big\|f|_{[\gamma]}\big\|$ is continuous,
note that $\|f|_{[\gamma]}\| =  \|(f^**f)|_{[s(\gamma)]}\|^{1/2}$.
Since $[x] \mapsto \big\|(f^**f)|_{[x]}\big\|$ is
continuous on $\Hh_\Lambda^{(0)}$, and since the source map $[\gamma] \mapsto
[s(\gamma)]$ in $\Hh_\Lambda$ is continuous, it follows that $[\gamma] \mapsto
\big\|f|_{[\gamma]}\big\|$ is continuous.

It is straightforward to check that for each $b \in B^\circ_{[\gamma]}$, there exists $f
\in C_c(\Gg_\Lambda, \sigma)$ such that $b = f|_{[\gamma]}$.  Hence, for each $[\gamma]$,
$\{ f|_{[\gamma]} : f \in C_c(\Gg_\Lambda, \sigma) \}$ is dense in $B_{[\gamma]}$.
\end{proof}

We now establish that $C^*(\Lambda, c)$ can be identified with the $C^*$-algebra of the
Fell bundle we have just constructed. Recall that if $\sB$ is a Fell bundle over a
groupoid $\Gg$, then $C^*(\Gg; \sB)$, the $C^*$-algebra of the bundle, is a universal
completion of the algebra of compactly supported sections of the bundle, and $C^*_r(\Gg;
\sB)$ is the corresponding reduced $C^*$-algebra.

\begin{thm}\label{thm:bundleiso}
Suppose that $\Lambda$ is a row-finite cofinal $k$-graph with no sources, and take $c \in
\Zcat2(\Lambda, \TT)$. Let $\Hh_\Lambda$ denote the quotient of $\Gg_\Lambda$ by the
interior $\Ii_\Lambda$ of its isotropy, and let $\sB_\Lambda$ be the Fell bundle over
$\Hh_\Lambda$ described in Proposition~\ref{prp:bundle}. Then $C^*(\Hh_\Lambda;
\sB_\Lambda) = C_r^*(\Hh_\Lambda; \sB_\Lambda)$, and there is an isomorphism $\pi :
C^*(\Lambda, c) \cong C^*(\Hh_\Lambda; \sB_\Lambda)$ such that $\pi(s_\lambda)([\gamma])
= 1_{Z(\lambda, s(\lambda))}|_{[\gamma]}$ for all $\lambda \in \Lambda$ and $\gamma \in
\Gg_\Lambda$.
\end{thm}
\begin{proof}
Corollary~\ref{cor:ILambda} says that $\Hh_\Lambda$ is amenable. Hence $C^*(\Hh_\Lambda;
\sB_\Lambda) = C^*_r(\Hh_\Lambda; \sB_\Lambda)$ by \cite[Theorem~1]{SimWil:IJMxx}. Define
elements $t_\lambda$ of $C_c(\Hh_\Lambda; \sB_\Lambda)$ by $t_\lambda([\gamma]) :=
1_{Z(\lambda, s(\lambda))}|_{[\gamma]} \in B^\circ_{[\gamma]} \subseteq B_{[\gamma]}$.
Theorem~6.7 of \cite{KPS4} shows that the $1_{Z(\lambda, s(\lambda))}$ form a
Cuntz--Krieger $(\Lambda, c)$-family in $C^*(\Gg_\Lambda, \sigma)$. It follows that the
$t_\lambda$ constitute a Cuntz--Krieger $(\Lambda, c)$-family in $C^*(\Hh_\Lambda;
\sB_\Lambda)$. So the universal property of $C^*(\Lambda, c)$ yields a homomorphism $\pi
: C^*(\Lambda, c) \to C^*(\Hh_\Lambda; \sB_\Lambda)$.

To see that $\pi$ is injective, we aim to apply the gauge-invariant uniqueness theorem
\cite[Corollary~7.7]{KPS4}. The projections $\{t_v : v \in \Lambda^0\}$ are nonzero
because the $Z(v)$ are nonempty, so we just need to show that there is an action $\beta$
of $\TT^k$ on $C^*(\Hh_\Lambda; \sB_\Lambda)$ such that $\beta_z(t_\lambda) =
z^{d(\lambda)}t_\lambda$ for all $\lambda \in \Lambda$.  Let $\tilde{d} : \Gg_\Lambda \to
\ZZ^k$ be the canonical cocycle $(x, m, y) \mapsto m$. For $z \in \TT^k$, $[\gamma] \in
\Hh_\Lambda$ and $f \in B^\circ_{[\gamma]}$ define $\beta^{[\gamma]}_z(f) \in
B^\circ_{[\gamma]}$ by $\beta^{[\gamma]}_z(f)(\alpha) = z^{\tilde{d}(\alpha)} f(\alpha)$.
Simple calculations show that $\beta_z^{[\gamma]}(f)*\beta_z^{[\gamma']}(g) =
\beta_z^{[\gamma\gamma']}(f * g)$ and that $\beta_z^{[\gamma^{-1}]}(f^*) =
\beta_z^{[\gamma]}(f)^*$. For $x \in \Gg_\Lambda^{(0)}$, the map $\beta_z^{[x]} :
C_c(\Ii_\Lambda, \sigma) \to C_c(\Ii_\Lambda, \sigma)$ extends to the canonical action of
$\TT^k$ on $A_x = C^*(\Per(\Lambda), \omega)$, and so is isometric. It follows that the
$\beta_z^{[\gamma]}$ are isometric for the norms on the fibres $B_{[\gamma]}$ of
$\sB_\Lambda$. For $f \in C_c(\Gg_\Lambda, \sigma)$ supported on a basic open set
$Z(\mu,\nu)$, the map $[\gamma] \mapsto \beta_z^{[\gamma]}(f)$ is clearly continuous.
Since such $f$ span $C_c(\Gg_\Lambda, \sigma)$, and by definition of the topology on
$\sB$ (see Proposition~\ref{prp:bundle}), it follows that if $p : \sB_\Lambda \to
\Hh_\Lambda$ is the bundle map, then $\xi \mapsto \beta_z^{p(\xi)}(\xi)$ is continuous.
So for fixed $z \in \TT^k$, the collection $\beta^{[\gamma]}_z$ determines an
automorphism of the bundle $\sB_\Lambda$, and hence induces an automorphism $\beta_z$ of
$C^*(\Hh_\Lambda; \sB_\Lambda)$. It is routine to check that $z \mapsto \beta_z$ is an
action of $\TT^k$ on $C^*(\Hh_\Lambda; \sB_\Lambda)$ such that $\beta_z(f)([\gamma]) =
\beta^{[\gamma]}_z(f([\gamma]))$ for $f \in C_c(\Hh_\Lambda; \sB_\Lambda)$ and $z \in
\TT^k$. In particular, each $\beta_z(t_\lambda) = z^{d(\lambda)}t_\lambda$, and so the
gauge-invariant uniqueness theorem \cite[Corollary 7.7]{KPS4} shows that $\pi$ is
injective as required.

It remains to show that $\pi$ is surjective. For this, it suffices to fix $\gamma,
\gamma' \in \Gg_\Lambda$ such that $[\gamma] \not= [\gamma']$, and show that the set
$\{\pi(a)([\gamma]) : a \in C^*(\Lambda ,c), \pi(a)[\gamma'] = 0\}$ is dense in
$B_{[\gamma]}$. Since $\pi$ is linear, it will suffice to show that for each $\alpha \in
[\gamma]$ there exists $a \in C^*(\Lambda , c)$ such that $\pi(a)([\gamma']) = 0$ and
$\pi(a)([\gamma]) = \delta_\alpha$. Fix $\alpha \in [\gamma]$, say $\tilde{d}(\alpha) =
m$. Choose a compact open bisection $U \subseteq \tilde{d}^{-1}(m)$. If $m \not\in
\Per(\Lambda) + \tilde{d}(\gamma')$, then $U \cap [\gamma'] = \emptyset$; otherwise,
since $[\gamma'] \not= [\alpha]$, either $r(\gamma') \not= r(\alpha)$ or $s(\gamma')
\not= s(\alpha)$, and so we may shrink $U$ to ensure that $U \cap [\gamma'] = \emptyset$.
By definition of the topology on $\Gg_\Lambda$, there exist $\mu,\nu \in \Lambda$ such
that $\alpha \in Z(\mu,\nu) \subseteq U$. The function $t_\mu t^*_\nu$ takes values in
$\TT$ on $Z(\mu,\nu)$, and so there is a complex scalar $z \in \TT$ such that $a := z
s_\mu s^*_\nu$ satisfies $\pi(a)([\gamma]) = \delta_\alpha$ and $\pi(a)([\gamma']) = 0$
as claimed.
\end{proof}

To prove Theorem~\ref{thm:main simple}, we identify the action of $\Hh_\Lambda$ on
$\operatorname{Prim}(C^*(\Hh_\Lambda^{(0)}; \sB_\Lambda))$ obtained from~\cite{IW} with
the action $\theta$ of Lemma~\ref{lem:tilde-r cocycle}. To do this, we first give an
explicit description of the action from~ \cite{IW}.

\begin{lem}\label{lem:unitaries implement}
Let $\sB$ be a Fell bundle over a groupoid $\Gg$ with unital fibres over $\Gg^{(0)}$. Let
$\gamma \in \Gg$ and suppose that $u \in B_\gamma$ is unitary in the sense that $u^*u =
1_{A_{s(\gamma)}}$ and $uu^* = 1_{A_{r(\gamma)}}$. For any ideal $I$ of $A_{s(\gamma)}$,
we have $uIu^* = \clsp\{xay^* : x,y \in B_\gamma, a \in I\}$. In particular, $I \mapsto
uIu^*$ is the map from $\Prim(A_{s(\gamma)})$ to $\Prim(A_{r(\gamma)})$ described in
\cite[Lemma~2.1]{IW}.
\end{lem}
\begin{proof}
We clearly have $uIu^* \subseteq \clsp\{xay^* : x,y \in B_\gamma, a \in I\}$. For the
reverse inclusion, fix $x_i, y_i \in B_\gamma$ and $a_i \in I$, and observe that
\[
\sum_i x_i a_i y_i^*
	= \sum_i uu^* x_i a y_i uu^*
	= u \Big(\sum_i u^* x_i a y_i u\Big) u^*
	\in u I u^*.\qedhere
\]
\end{proof}

We can now identify the action described in \cite{IW} with that of Lemma~\ref{lem:tilde-r
cocycle}.

\begin{thm}\label{thm:action}
Let $\Lambda$ be a row-finite cofinal $k$-graph with no sources, and suppose that $c \in
\Zcat2(\Lambda, \TT)$. Take $\omega$ as in Proposition~\ref{prp:constant in Ii}, and let
$\Hh_\Lambda$ denote the quotient of $\Gg_\Lambda$ by the interior $\Ii_\Lambda$ of its
isotropy. Let $\sB_\Lambda$ be the Fell bundle over $\Hh_\Lambda$ described in
Proposition~\ref{prp:bundle}, and let $\sA_\Lambda$ denote the bundle of $C^*$-algebras
obtained by restricting $\sB_\Lambda$ to $\Hh_\Lambda^{(0)}$. There is a homeomorphism $i
: \operatorname{Prim}(C^*(\Hh_\Lambda^{(0)}; \sB_\Lambda)) \to \Lambda^\infty \times
\widehat{Z_\omega}$ that intertwines the action $\theta$ of $\Hh_\Lambda$ on
$\Lambda^\infty \times \widehat{Z}_\omega$ described in Lemma~\ref{lem:tilde-r cocycle}
and the action of $\Hh_\Lambda$ on $\operatorname{Prim}(C^*(\Hh_\Lambda^{(0)};
\sB_\Lambda))$ described in \cite{IW}.
\end{thm}
\begin{proof}
Put $H := \Per(\Lambda)/Z_\omega$. Take $x \in \Lambda^\infty$. Then
\[
A_x = C^*(\Per(\Lambda), \omega) \cong \clsp\{\delta_{(x,p,x)} : p \in \Per(\Lambda)\},
\]
and there is an isomorphism $i_x : A_x \to \Ind^{\Per(\Lambda)\widehat{\;}}_{\widehat{H}}
C^*(H, \omega)$ such that
\[
i_x(\delta_{(x,p,x)})(\chi) = \chi(p)U_{p + Z_\omega}\quad\text{ for all $\chi \in \Per(\Lambda)\widehat{\;}$.}
\]
In particular $i_x$ is equivariant for the action of $\Per(\Lambda)\widehat{\;}$ on $A_x$
given by $\chi \cdot \delta_{(x,p,x)} = \chi(p)\delta_{(x,p,x)}$ and the action $\lt$ of
$\Per(\Lambda)\widehat{\;}$ on $\Ind^{\Per(\Lambda)\widehat{\;}}_{\widehat{H}} C^*(H,
\omega)$ by translation.

Take $\gamma \in \Gg_\Lambda$ and $p \in \Per (\Lambda)$. A straightforward calculation
in $\sB_\Lambda$ shows that
\[
\delta_\gamma * \delta_{(s(\gamma), p, s(\gamma))} * \delta_\gamma^*
	= r^\sigma_\gamma(p)\delta_{(r(\gamma),p,r(\gamma))}.
\]
The element $\delta_\gamma$ is a unitary in the fibre $B_{[\gamma]}$, and so
Lemma~\ref{lem:unitaries implement} shows that conjugation by $\delta_\gamma$ in
$\sB_\Lambda$ implements the homeomorphism $\alpha_{[\gamma]}$ of \cite{IW} from
$\Prim(A_{s(\gamma)}) \subseteq \Prim(C^*(\Hh_\Lambda^{(0)}; \sB_\Lambda))$ to
$\Prim(A_{r(\gamma)})$. We have $i_{r(\gamma)} \circ \Ad_{\delta_\gamma} \circ
i_{r(\gamma)}^{-1} = \lt_{r^\sigma_\gamma}$, and it follows that for $\chi \in
\Per(\Lambda)\widehat{\;}$ the primitive ideal $I_{\chi \widehat{H}}$ of
$\Ind^{\Per(\Lambda)\widehat{\;}}_{\widehat{H}} C^*(H, \omega)$ consisting of functions
that vanish on $\chi\widehat{H}$ satisfies
\[
\big(i_{r(\gamma)} \circ \Ad_{\delta_\gamma} \circ i_{r(\gamma)}^{-1}\big)(I_{\chi \widehat{H}})
	= I_{r^\sigma_\gamma \cdot \chi \widehat{H}}.
\]
That is, the induced map $(i_{r(\gamma)})_* \alpha_{[\gamma]} (i_{r(\gamma)})_*^{-1}$ on
$\Prim\big(\Ind^{\Per(\Lambda)\widehat{\;}}_{\widehat{H}} C^*(H, \omega)\big)$ is
translation by $\tilde{r}^\sigma_{[\gamma]}$. Lemma~\ref{lem:twisted Per} yields a
homeomorphism $i : \Prim(C^*(\Hh^{(0)}; \sB_\Lambda )) \to \Lambda^\infty \times
\widehat{Z}_\omega$ such that $i|_{\Prim(A_x)} = i_x$, and this homeomorphism does the
job.
\end{proof}

Recall from \cite[Proposition~3.6]{Kumjian98} that if $\sB$ is a Fell bundle over an
\'etale\footnote{The statement in \cite{Kumjian98} says ``$r$-discrete,'' but \'etale is
meant.} groupoid $\Hh$, then restriction of compactly supported sections to $\Hh^{(0)}$
extends to a faithful conditional expectation $P : C^*_r(\Hh; \sB) \to C^*(\Hh^{(0)};
\sB)$.

\begin{lem}\label{lem:bundle uniqueness}
Suppose that $\Hh$ is an \'etale topologically principal groupoid and that $\sB$ is a
Fell bundle over $\Hh$. If $I$ is a nonzero ideal of $C^*_r(\Hh; \sB)$, then $I \cap
C^*(\Hh^{(0)}; \sB) \not= \{0\}$.
\end{lem}
\begin{proof}
Our argument follows that of \cite[Proposition~2.4]{A-D} (see also \cite[Lemma~4.2]{BCFS}
or \cite[Lemma~3.5]{KPR}).

Let $I$ be a nonzero ideal of $C^*_r( \Hh; \sB)$. Let $P : C^*_r(\Hh; \sB) \to
C^*(\Hh^{(0)}; \sB)$ be the faithful conditional expectation discussed above. Choose $a
\in I^+$ such that $\|P(a)\| = 1$. Choose $b \in \Gamma_c(\Hh; \sB) \cap C^*_r(\Hh;
\sB)^+$ such that $\|a - b\| < 1/4$, so that $\|P(b)\| > 3/4$. Then $b - P(b) \in
\Gamma_c(\Hh; \sB)$; thus, $K := \operatorname{supp}(b - P(b))$ is compact and contained
in $\Hh \setminus \Hh^{(0)}$. Let $U = \{u \in \Hh^{(0)} : \|P(b)(u)\| > 3/4\}$. By
\cite[Lemma~3.3]{BCFS}, there exists an open set $V$ such that $V \subseteq U$ and $VKV =
\emptyset$.

Fix a continuous function $h : \RR\to [0,1]$ such that $h$ is identically 0 on $(-\infty,
1/2]$ and identically 1 on $[3/4, \infty)$. Then
\[
h(P(b))\, P(b)\, h(P(b)) \ge \frac12 h(P(b))^2.
\]
Choose $g \in C_c(\Hh^{(0)}) \subseteq \Mm C^*_r(\Hh; \sB)$ such that $\|g\|_\infty = 1$
and $\operatorname{supp}(g) \subseteq V$. Let $f := g h(P(b))$. Since $\|h(P(b))(u)\| =
1$ for all $u \in U$, and since $V \subseteq U$, we have $f \not= 0$.

Since $\operatorname{supp}(f) \subseteq \operatorname{supp}(g) \subseteq V$, we have $f
(b - P(b)) f = 0$. Hence
\begin{align*}
fbf &= f\big(P(b) + (b - P(b))\big)f\\
	&= fP(b)f
	= g h(P(b))\, P(b)\, h(P(b)) g
	\ge  \frac12 g^2 h(P(b))^2
	= \frac12 f^2.
\end{align*}
Thus
\[
faf \ge fbf - \frac14f^2 = fP(b)f - \frac14f^2 \ge \frac14f^2.
\]
Since $faf \in I$ and $I$ is hereditary, we deduce that $\frac14f^2 \in I \cap
C^*(\Hh^{(0)}; \sB) \setminus \{0\}$.
\end{proof}

Lemma~\ref{lem:bundle uniqueness} allows us to characterise the simplicity of $C^*(\Hh;
\sB)$ when $\Hh$ is amenable and topologically principal.

\begin{cor}\label{cor:minimality}
Suppose that $\Hh$ is a topologically principal amenable groupoid and that $\sB$ is a Fell bundle over $\Hh$.
Then $C^*(\Hh; \sB)$ is simple if and only if the action of $\Hh$ on $\Prim(C^*(\Hh^{(0)}; \sB))$ described in \cite[Section~2]{IW} is minimal.
\end{cor}
\begin{proof}
We denote the action of $\Hh$ on $\Prim(C^*(\Hh^{(0)}; \sB))$ by $\vartheta$. First
suppose that $\vartheta$ is not minimal. Then there exists $\rho$ in
$\operatorname{Prim}(C^*(\Hh^{(0)}; \sB))$ such that $\Hh \cdot\rho$ is not dense in
$\operatorname{Prim}(C^*(\Hh^{(0)}; \sB))$. Hence $X := \overline{\Hh \cdot \rho}$ is a
nontrivial closed invariant subspace of $\operatorname{Prim}(C^*(\Hh^{(0)}; \sB))$. The
set $I_X$ of sections of $\left.\sB\right|_{\Hh^{(0)}}$ that vanish on $X$ is a
nontrivial $\Hh_\Lambda$-invariant ideal of the $C^*$-algebra $C^*(\Hh^{(0)}; \sB)$
sitting over the unit space of $\Hh$. Thus the ideal of $C^*(\Hh; \sB)$ generated by
$I_X$ is a proper nontrivial ideal \cite[Theorem~3.7]{IW} and so $C^*(\Hh; \sB)$ is not
simple.

Now suppose that $\vartheta$ is minimal. Let $I$ be a nonzero ideal of $C^*(\Hh; \sB)$.
Theorem~1 of \cite{SimWil:IJMxx} shows that $C^*(\Hh; \sB) = C^*_r(\Hh; \sB)$, so
Lemma~\ref{lem:bundle uniqueness} implies that $I_0 := I \cap C^*(\Hh^{(0)}; \sB)$ is
nonzero. Let $X$ denote the set of primitive ideals of $C^*(\Hh^{(0)}; \sB)$ that contain
$I_0$. Since $I$ is nonzero, $X$ is nonempty, and it is closed by definition of the
topology on $\operatorname{Prim}(C^*(\Hh^{(0)}; \sB))$. Lemma~2.1 of \cite{IW} implies
that $X$ is also $\vartheta$-invariant. Since $\vartheta$ is minimal, it follows that $X
= \operatorname{Prim}(C^*(\Hh^{(0)}; \sB))$, and so $I_0 = C^*(\Hh^{(0)}; \sB)$. Since
$C^*(\Hh^{(0)}; \sB)$ contains an approximate identity for $C^*_r(\Hh; \sB)$, it follows
that $I = C^*(\Hh; \sB)$ as required.
\end{proof}

Since we established in Corollary~\ref{cor:ILambda} that $\Hh_\Lambda$ is always
topologically principal and amenable, we obtain an immediate corollary.

\begin{cor}\label{cor:simple<->minimal}
Let $\Lambda$ be a row-finite cofinal $k$-graph with no sources, and suppose that $c \in
\Zcat2(\Lambda, \TT)$. Let $\theta$ be the action of $\Hh_\Lambda$ on $\Lambda^\infty
\times \Per(\Lambda)\widehat{\;}$ obtained from any choice of $\sigma$ in
Lemma~\ref{lem:tilde-r cocycle}. Then $C^*(\Lambda, c)$ is simple if and only if $\theta$
is minimal.
\end{cor}
\begin{proof}
Theorem~\ref{thm:bundleiso} shows that $C^*(\Lambda,c)$ is simple if and only if
$C^*(\Hh_\Lambda; \sB_\Lambda)$ is simple. Corollary~\ref{cor:ILambda} shows that
$\Hh_\Lambda$ is amenable and topologically principal, and so
Corollary~\ref{cor:minimality} implies that $C^*(\Hh_\Lambda; \sB_\Lambda)$ is simple if
and only if the action of $\Hh_\Lambda$ on $\Prim(C^*(\Hh_\Lambda^{(0)}; \sB_\Lambda))$
is minimal. Theorem~\ref{thm:action} shows that this action is minimal if and only if
$\theta$ is minimal, giving the result.
\end{proof}

We can now prove our main theorem.

\begin{proof}[Proof of Theorem~\ref{thm:main simple}]
By Corollary~\ref{cor:simple<->minimal}, it is enough to show that the set
\begin{equation}\label{eq:mustbedense}
\{(r(\gamma), r^\sigma_\gamma|_{Z_\omega}) : \gamma \in (\Gg_\Lambda)_x\}
\end{equation}
is dense in $\Lambda^\infty \times \widehat{Z}_\omega$ for every $x \in \Lambda^\infty$
if and only if the action $\theta$ of $\Hh_\Lambda$ on $\Lambda^\infty \times
\Per(\Lambda)\widehat{\;}$ is minimal. Clearly~\eqref{eq:mustbedense} is dense for every
$x$ if and only if $\{(r(\gamma), r^\sigma_\gamma|_{Z_\omega} \cdot \chi) : \gamma \in
(\Gg_\Lambda)_x\}$ is dense in $\Lambda^\infty \times Z_\omega$ for every $x$ and every
$\chi \in \widehat{Z}_\omega$, which is precisely minimality of $\theta$.
\end{proof}

\section{Twists induced by torus-valued 1-cocycles, and crossed-products by quasifree actions}\label{sec:CP}

In this section we describe a class of examples of twisted $(k+l)$-graph $C^*$-algebras
arising from $\TT^l$-valued $1$-cocycles on aperiodic $k$-graphs. We describe the
simplicity criterion obtained from our main theorem for these examples. We then show that
these twisted $C^*$-algebras can also be interpreted as twisted crossed-products of
$k$-graph algebras by quasifree actions, giving a characterisation of simplicity for the
latter. See also \cite[Theorem~2.1]{BDKP} for the case where $l = 1$.

Recall that we write $T_l$ for $\NN^l$ when regarded as an $l$-graph with degree map the
identity functor.

Given a cocycle $\omega \in Z^2(\ZZ^l, \TT)$ and an action $\alpha$ of $\ZZ^l$ on a
$C^*$-algebra $A$, we write $A \times_{\alpha,\omega} \ZZ^l$ for the twisted crossed
product, which is the universal $C^*$-algebra generated by unitary multipliers $\{u_n \in
\Mm(A \times_{\alpha,\omega} \ZZ^l) : n \in \ZZ^l\}$ and a homomorphism $\pi : A \to A
\times_{\alpha,\omega} \ZZ^l$ such that $u_m \pi(a) u_m^* = \pi(\alpha_m(a))$ for all
$m,a$ and such that $u_m u_n = \omega(m,n)u_{m+n}$ for all $m,n$. For further details on
twisted crossed products, see \cite{PackerRaeburn}.

Let $\Lambda$ be a row-finite $k$-graph with no sources. Consider a $1$-cocycle $\phi \in
\Zcat1(\Lambda, \TT^l)$. Define a $2$-cocycle $c_\phi \in \Zcat2(\Lambda \times T_l,
\TT)$ by
\begin{equation}\label{eq:phic}
c_\phi((\lambda, m), (\mu,n)) = \phi(\mu)^m \quad \text{ for composable $\lambda,\mu \in \Lambda$ and $m,n \in \NN^l = T_l$.}
\end{equation}
There is a continuous cocycle $\tilde{\phi} \in Z^1(\Gg_\Lambda, \TT^l)$ such that
\[
\tilde{\phi}(\mu x, d(\mu) - d(\nu), \nu x) = \phi(\mu)\overline{\phi(\nu)}\quad\text{ for all $x \in \Lambda^\infty$ and $\mu,\nu \in \Lambda r(x)$.}
\]
Let $\omega$ be a bicharacter of $\ZZ^l$. There is a cocycle $c_{\phi, \omega}\in \Zcat2
(\Lambda \times T_l, \TT)$ such that
\begin{equation}\label{eq:phicomega}
c_{\phi, \omega}((\lambda, m), (\mu,n)) = \phi(\mu)^m \omega(m,n).
\end{equation}

The next theorem is the main result of this section. It characterises simplicity of
$C^*(\Lambda \times T_l, c_{\phi,\omega})$ in terms of $\tilde{\phi}$ and the centre
$Z_\omega$ of $\omega$ described in~\eqref{eq:Zomega}, under the simplifying assumption
that $\Lambda$ is aperiodic.

\begin{thm}\label{thm:easytwistsimplicity}
Let $\Lambda$ be a row-finite $k$-graph with no sources. Take $\phi \in \Zcat1(\Lambda,
\TT^l)$, and let $\omega$ be a bicharacter of $\ZZ^l$. Let $\tilde{\phi} \in
Z^1(\Gg_\Lambda, \TT)$ and $c_{\phi,\omega} \in \Zcat2(\Lambda \times T_l, \TT)$ be as
above. Then
\begin{enumerate}
\item \label{one} there is an action $\beta$ of $\ZZ^l$ on $C^*(\Lambda)$ such that
    $\beta_m(s_\lambda) = \phi(\lambda )^m s_\lambda$ for all $\lambda \in \Lambda$
    and $m \in \ZZ^l$;
\item \label{two} there is an isomorphism $\rho : C^* (\Lambda \times T_l,
    c_{\phi,\omega}) \cong C^*(\Lambda) \times_{\beta, \omega} \ZZ^l$ such that
\[
\rho ( s_{(\lambda,m)}) = \pi(s_\lambda) u_m \text{ for all $\lambda \in \Lambda$ and $m \in T_l$; and }
\]
\item \label{three} if $\Lambda$ is aperiodic, then $C^*(\Lambda \times T_l,
    c_{\phi,\omega})$ is simple if and only if the orbit $\{(r(\gamma),
    \tilde{\phi}(\gamma)|_{Z_\omega}) : \gamma \in (\Gg_\Lambda)_x \}$ is dense in
    $\Lambda^\infty \times \widehat{Z}_\omega$ for all $x \in \Lambda^\infty$.
\end{enumerate}
\end{thm}

\begin{rmk}\label{rmk:observant reader}
The observant reader may be surprised to note that there is no cofinality hypothesis on
the preceding theorem. But a moment's reflection shows that it is still there, just
hidden: each of the conditions that $C^*(\Lambda \times T_l, c_{\phi, \omega})$ is
simple, and that each $\{(r(\gamma), \tilde{\phi}(\gamma)|_{Z_\omega}) : \gamma \in
(\Gg_\Lambda)_x\}$ is dense in $\Lambda^\infty \times \widehat{Z}_\omega$ implies that
$\Lambda$ is cofinal.
\end{rmk}

We prove (\ref{one})~and~(\ref{two}) here and defer the proof of~(\ref{three}) to the end
of the section.

\begin{proof}[Proof of Theorem~\ref{thm:easytwistsimplicity}(\ref{one})~and~(\ref{two})]
For part~(\ref{one}) observe that for each $m \in \ZZ^l$ the set $\{ \phi(\lambda )^m
s_\lambda : \lambda \in \Lambda \}$ is a Cuntz--Krieger $\Lambda$-family, so induces a
homomorphism $\beta_m : C^*(\Lambda) \to C^*(\Lambda)$ such that $\beta_m(s_\lambda) =
\phi(\lambda )^m s_\lambda$. Clearly $\beta_0 = \id$ and $\beta_m \circ \beta_n =
\beta_{m+n}$, so $\beta$ is an action.

For part~(\ref{two}) we first check that $\{\pi(s_\lambda)u_m : (\lambda,m) \in \Lambda
\times T_l\}$ is a Cuntz--Krieger $( \Lambda \times T_l , c_{\phi, \omega})$-family. For
$\lambda \in \Lambda$ and $m \in T_l$ let $t_{(\mu,m)} = \pi(s_\lambda) u_m$. It is easy
to check that~(CK1), (CK3)~and~(CK4) hold because each $\beta_m$ fixes each $s_\lambda
s_\lambda^*$. To check~(CK2) we compute
\begin{align*}
t_{(\lambda, m)} t_{(\mu, n)}
    &= \pi(s_\lambda) u_m  \pi(s_\mu)u_n
    = \pi(s_\lambda) u_m \pi(s_\mu) u^*_m u_m u_n \\
    &= \pi(s_\lambda) \pi(\beta_m(s_\mu)) \omega(m,n) u_{m+n}
    = \phi(\mu)^m\omega(m,n) \pi(s_{\lambda \mu}) u_{m+n} \\
    &= \phi(\mu)^m\omega(m,n) t_{(\lambda\mu, m+n)}
    = c_{\phi,\omega}((\lambda, m) (\mu, n)) t_{(\lambda,m)(\mu,n)}.
\end{align*}

Now the universal property of $C^* ( \Lambda \times T_l , c_{\phi,\omega} )$ gives a
homomorphism $\rho$ satisfying the desired formula. The gauge-invariant uniqueness
theorem \cite[Corollary~7.7]{KPS4} shows that $\rho$ is injective. The map $\rho$ is
surjective because its image contains all the generators of the twisted crossed product.
\end{proof}

In order to prove~(\ref{three}) we first do some preparatory work to identify the action
$\theta$ of Lemma~\ref{lem:tilde-r cocycle} in terms of the cocycle $\tilde{\phi} \in
Z^2(\Gg_\Lambda , \TT)$. Before that, though, a comment on the hypotheses of
Theorem~\ref{thm:easytwistsimplicity} is in order.

\begin{rmk}
Note that the aperiodicity hypothesis in
Theorem~\ref{thm:easytwistsimplicity} is needed in our proof and simplifies the
statement, but is not a necessary condition for $C^*(\Lambda \times T_l,
c_{\phi,\omega})$ to be simple.  To see this, consider $\Lambda = T_1$, which is $\NN$
regarded as a $1$-graph. Put $l = 1$, define $\phi : \Lambda \to \TT$ by $\phi(1) =
e^{2\pi i\theta}$ and take $\omega$ to be trivial. Then $\Lambda$ is cofinal, but
certainly not aperiodic. We have $\Lambda \times T_1 \cong T_2$, and $c_{\phi,\omega} \in
\Zcat2(T_2, \TT)$ is given by $c_{\phi,\omega}(m,n) := e^{2\pi i \theta m_2 n_1}$, so
$C^*(\Lambda \times T_1, c_{\phi,\omega})$ is the rotation algebra $A_\theta$, which is
simple whenever $\theta$ is irrational.
\end{rmk}

Let $\Gamma := \Lambda \times T_l$.  The proof of Theorem~\ref{thm:easytwistsimplicity}
(3) requires quite a bit of preliminary work. To begin preparations, observe that the
projection map $\pi : \Gamma \to \Lambda$ given by $(\lambda, m) \mapsto \lambda$ induces
a homeomorphism
\[
\pi_\infty : \Gamma^\infty \to \Lambda^\infty
    \quad\text{such that $\pi_\infty(x)(m,n) = \pi(x((m,0), (n,0)))$.}
\]
There is an isomorphism $\Gg_\Gamma \cong \Gg_\Lambda \times \ZZ^l$ given by
\begin{equation}\label{eq:GGamma cong GLambdaxZl}
\begin{split}
\big((\alpha, m) x,& (d(\alpha),m)-(d(\beta),n), (\beta, n)x\big)\\
    &\mapsto \big((\alpha \pi_\infty(x), d(\alpha) - d(\beta), \beta\pi_\infty(x)), m-n\big).
\end{split}
\end{equation}

Suppose that $\Lambda$ is cofinal and aperiodic, and fix a subset $\Pp$ of $\Lambda
\mathbin{{_s*_s}} \Lambda$ such that $(\mu,s(\mu)) \in \Pp$ for all $\mu \in \Lambda$ and
such that $\{Z(\mu,\nu) : (\mu,\nu) \in \Pp\}$ is a partition of $\Gg_\Lambda$. For $g
\in \Gg_\Lambda$, let $(\mu_g, \nu_g) \in \Pp$ be the unique pair such that $g \in
Z(\mu_g, \nu_g)$. Observe that $\Per(\Gamma) = \{0\} \times \ZZ^l \subseteq \ZZ^{k+l}$.
Recall that $\Hh_\Gamma$ is the quotient $\Gg_\Gamma/\Ii_\Gamma$ of the $(k+l)$-graph
groupoid by the interior of its isotropy.

\begin{prop}\label{prp:action formula}
Let $\Lambda$ be a row-finite cofinal aperiodic $k$-graph with no sources, and let
$\Gamma := \Lambda \times T_l$. Then $\Per(\Gamma) = \{0\} \times \ZZ^l \subseteq
\ZZ^{k+l}$, the identification of $\Gg_\Gamma$ with $\Gg_\Lambda \times \ZZ^l$ described
above carries $\Ii_\Gamma$ to $\Gg_\Lambda^{(0)} \times \ZZ^l$, and $(g,m)\cdot\Ii_\Gamma
\mapsto g$ gives an isomorphism $\Hh_\Gamma \cong \Gg_\Lambda$. Let $\phi : \Lambda \to
\TT^l$ be a $1$-cocycle, and let $\omega$ be a bicharacter of $\ZZ^l$. Let
$c_{\phi,\omega} \in \Zcat2(\Gamma,\TT)$ be the 2-cocycle of~\eqref{eq:phicomega}. There
exists a cocycle $\sigma \in Z^2(\Gg_\Gamma, \TT)$ such that $\sigma|_{\Ii_\Gamma} =
1_{\Lambda^\infty} \times \omega$, $C^*(\Gg_\Gamma, \sigma) \cong C^*(\Gamma, c_{\phi,
\omega})$ and the action $\theta$ of $\Gg_\Lambda$ on $\Lambda^\infty \times
\widehat{Z}_\omega$ of Lemma~\ref{lem:tilde-r cocycle} satisfies
\begin{equation}\label{eq:theta formula}
    \theta_g(s(g), \chi) = (r(g), \overline{\tilde{\phi}(g)|_{Z_\omega}}\chi).
\end{equation}
\end{prop}

Before proving Proposition~\ref{prp:action formula}, we establish two technical lemmas
that will help in the proof. The first of these shows that the passage from a $k$-graph
cocycle to a groupoid cocycle described in~\eqref{eq:sigmac} has no effect on cohomology
when the $k$-graph in question is $\NN^l$ and its groupoid is $\ZZ^l$.

Recall that for $m \in \ZZ^l$, we write $m^+ := m \vee 0$ and $m^- := (-m) \vee 0$; so $m
= m^+ - m^-$ and $m^+ \wedge m^- = 0$.

\begin{lem}\label{lem:technical}
Let $\omega$ be a bicharacter of $\ZZ^l$. Then $\omega|_{T_l \times T_l}$ belongs to
$\Zcat2(T_l, \TT)$. The $l$-graph $T_l$ has a unique infinite path $x$, and there is an
isomorphism $\Gg_{T_l} \cong \ZZ^l$ given by $(x,m,x) \mapsto m$. Let $\Pp = \{(m^+, m^-)
: m \in \ZZ^l\}$. Then $(\lambda, s(\lambda)) \in \Pp$ for all $\lambda \in T_l$, and
$\Gg_{T_l} = \bigsqcup_{(\mu,\nu)\in \Pp} Z(\mu,\nu)$. Let $\sigma_\omega \in
Z^2(\Gg_{T_l}, \TT)$ be the 2-cocycle obtained from $\omega$ and $\Pp$ as
in~\eqref{eq:sigmac}. Then $\sigma_\omega$ is cohomologous to $\omega$ when regarded as a
cocycle on $\ZZ^l$.
\end{lem}
\begin{proof}
Every bicharacter of $\ZZ^l$ is a 2-cocycle, so $\omega$ restricts to a cocycle on $T_l$.
It is clear that $\Gg_{T_l} \cong \ZZ^l$ as claimed. We identify $\Gg_{T_l}$ with $\ZZ^l$
for the rest of the proof.

For $m \in \Gg_{T_l}$, the pair $(\mu_m, \nu_m) = (m^+, m^-)$ is the unique element of
$\Pp$ such that $m \in Z(\mu_m, \nu_m)$. So for $m,n \in \Gg_{T_l}$,
equation~\eqref{eq:sigmac} gives
\[
\sigma_\omega(m,n)
    = \omega(m^+, \alpha)\overline{\omega(m^-, \alpha)}
        \omega(n^+, \beta)\overline{\omega(n^-, \beta)}
            \overline{\omega((m+n)^+, \gamma)}\omega((m+n)^-, \gamma)
\]
for any choice of $\alpha, \beta, \gamma \in T_l$ such that $m^+ + \alpha = (m+n)^+ +
\gamma$, $n^- + \beta = (m+n)^- + \gamma$ and $m^- + \alpha = n^+ + \beta$.

We show that $\sigma_{\omega}(e_i, e_j) = \omega(e_i, e_j)$ for all $i,j \le l$. For this
observe that for $m = e_i$ and $n = e_j$, the elements $\alpha = e_j$, $\beta = 0$ and
$\gamma = 0$ satisfy the above conditions, so
\[
\sigma_\omega(e_i,e_j)
    = \omega(e_i, e_j)\overline{\omega(0, e_j)}
        \omega(e_j, 0)\overline{\omega(0, 0)}
        \overline{\omega(e_i+e_j, 0)}\omega(0, 0)\\
    = \omega(e_i, e_j)
\]
as claimed.

Hence $(\sigma_\omega \sigma^*_\omega)(e_i, e_j) = (\omega\omega^*)(e_i, e_j)$ for all
$i,j$. These are bicharacters by \cite[Proposition~3.2]{OPT} and so we have
$\sigma_\omega \sigma^*_\omega = \omega\omega^*$. Thus \cite[Proposition~3.2]{OPT}
implies that $\sigma_\omega$ is cohomologous to $\omega$ as claimed.
\end{proof}

Our second technical result shows how to obtain a partition $\Qq$
satisfying~\eqref{eq:Pproperties} for the $(k+l)$-graph $\Lambda \times T_l$ from a
partition $\Pp$ satisfying~\eqref{eq:Pproperties} for $\Lambda$.

\begin{lem}\label{lem:partition}
Let $\Lambda$ be a row-finite $k$-graph with no sources. Suppose that $\Pp \subseteq
\Lambda \mathbin{{_s*_s}} \Lambda$ satisfies $(\mu, s(\mu)) \in \Pp$ for all $\mu$, and
$\Gg_\Lambda = \bigsqcup_{(\mu,\nu) \in \Pp} Z(\mu,\nu)$. Let $\Qq := \{((\mu,m^+),
(\nu,m^-)) : (\mu,\nu) \in \Pp, m \in \ZZ^l\}$. Then $(\alpha, s(\alpha)) \in \Qq$ for
all $\alpha \in \Gamma$, and $\Gg_\Gamma = \bigsqcup_{(\alpha,\beta) \in \Qq}
Z(\alpha,\beta)$.
\end{lem}
\begin{proof}
Consider an element $\alpha = (\mu,m) \in \Gamma$. We have $m^+ = m$ and $m^- = 0$, and
$(\alpha, s(\alpha)) = ((\mu,m), (s(\mu), 0))$. Since $(\mu, s(\mu)) \in \Pp$, we have
\[
((\mu,m), (s(\mu), 0)) = ((\mu, m^+), (s(\mu), m^-)) \in \Qq.
\]

Let $\rho : \Gg_\Gamma \to \Gg_\Lambda \times \ZZ^l$ denote the map~\eqref{eq:GGamma cong
GLambdaxZl}. Then
\begin{align*}
\Gg_\Gamma
    &= \rho^{-1}(\Gg_\Lambda \times \ZZ^l)
    = \bigsqcup_{(\mu,\nu) \in \Pp} \rho^{-1}(Z(\mu,\nu) \times \ZZ^l) \\
    &= \bigsqcup_{(\mu,\nu) \in \Pp} \bigsqcup_{m \in \ZZ^l}
        \rho^{-1}(Z(\mu,\nu) \times \{m\}) \\
    &= \bigsqcup_{(\mu,\nu) \in \Pp} \bigsqcup_{m \in \ZZ^l}
        Z((\mu, m^+),(\nu, m^-))
    = \bigsqcup_{(\alpha,\beta) \in \Qq} Z(\alpha,\beta).\qedhere
\end{align*}
\end{proof}

\begin{proof}[Proof of Proposition~\ref{prp:action formula}]
Since $\pi_\infty : \Gamma^\infty \to \Lambda^\infty$ intertwines the shift maps, we have
$\shift^{(p,m)}(x) = \shift^{(q,n)}(x)$ if and only if $\shift^p(\pi_\infty(x)) =
\shift^q(\pi_\infty(x))$. Hence $\Per(\Gamma) = \Per(\Lambda) \times \ZZ^l = \{0\} \times
\ZZ^l$ because $\Lambda$ is aperiodic. The next two assertions are straightforward to
check using the definition of the isomorphism $\Gg_\Gamma \cong \Gg_\Lambda \times
\ZZ^l$.

We identify $\Gg_\Gamma$ with $\Gg_\Lambda \times \ZZ^l$ for the remainder of this proof.
Choose $\Pp \subseteq \Lambda \mathbin{{_s*_s}} \Lambda$ such that $(\lambda,s(\lambda))
\in \Pp$ for all $\lambda$ and $\Gg_\Lambda = \bigsqcup_{(\mu,\nu) \in \Pp} Z(\mu,\nu)$.
For $g \in \Gg_\Lambda$, write $(\mu_g, \nu_g)$ for the element of $\Pp$ with $g \in
Z(\mu_g, \nu_g)$. Let $\Qq = \{(\mu, m^+), (\nu, m^-) : (\mu,\nu) \in \Pp, m \in \ZZ^l\}$
as in Lemma~\ref{lem:partition}. For $(g,m) \in \Gg_\Gamma$, let $\tilde{\mu}_{(g,m)} :=
(\mu_g,m^+)$, and $\tilde{\nu}_{(g,m)}) = (\nu_g, m^-)$. Then $(\tilde{\mu}_{(g,m)},
\tilde{\nu}_{(g,m)})$ is the unique element of $\Qq$ such that $(g,m) \in
Z(\tilde{\mu}_{(g,m)}, \tilde{\nu}_{(g,m)})$. For $((x,0,x), m) \in \Ii_\Gamma$, we have
$\tilde{\mu}_{((x,0,x),m)} = (r(x), m^+)$ and $\tilde{\nu}_{((x,0,x),m)} = (r(x), m^-)$.

Let $\tilde\omega$ denote the $2$-cocycle on $\Gamma$ given by $((\mu,m), (\nu,n))
\mapsto \omega(m,n)$, and recall that $c_\phi \in \Zcat2(\Lambda \times T_l, \TT)$ is
given by~\eqref{eq:phic}. Let $\sigma_{c_\phi}$ be the 2-cocycle on $\Gg_\Gamma$ obtained
from~\eqref{eq:sigmac} applied to $c_\phi \in \Zcat2(\Gamma, \TT)$ and $\Qq$, and let
$\sigma_{\tilde\omega}$ be the 2-cocycle on $\Gg_\Gamma$ obtained in the same way from
the cocycle $((\mu,m), (\nu,n)) \mapsto \omega(m,n)$ on $\Gamma$. Using that
$c_{\phi,\omega} = c_\phi \tilde\omega$ and the definitions~\eqref{eq:sigmac} of
$\sigma_{c_\phi}$, $\sigma_{\tilde\omega}$ and $\sigma_{c_{\phi,\omega}}$, it is easy to
see that $\sigma_{c_{\phi,\omega}} = \sigma_{c_\phi}\sigma_{\tilde\omega}$. Let
$\sigma_\omega \in Z^2(\Gg_{T_l}, \TT)$ be the cocycle obtained from $\omega$ as in
Lemma~\ref{lem:technical}. The formulas for $\sigma_{\tilde\omega}$ and $\sigma_\omega$
show that the identification $\Gg_\Gamma \cong \Gg_\Lambda \times \ZZ^l \cong \Gg_\Lambda
\times \Gg_{T_l}$ carries $\sigma_{\tilde\omega}$ to $1_{Z^2(\Gg_\Lambda, \TT)} \times
\sigma_\omega$. Thus Lemma~\ref{lem:technical} shows that $\sigma_{c_{\phi,\omega}}$ is
cohomologous to the cocycle $\sigma \in Z^2(\Gg_\Gamma, \TT)$ given by
\[
\sigma((g,m),(h,n)) = \sigma_{c_\phi}((g,m),(h,n)) \omega(m,n).
\]
Corollary~7.9 of \cite{KPS4} shows that $C^*(\Gamma, c_{\phi,\omega}) \cong
C^*(\Gg_\Gamma, \sigma_{c_{\phi,\omega}})$. Proposition~II.1.2 of \cite{Ren} says that
cohomologous groupoid cocycles determine isomorphic twisted groupoid $C^*$-algebras, and
so we have $C^*(\Gamma, c_{\phi,\omega}) \cong C^*(\Gg_\Gamma, \sigma)$. We have
$\sigma|_{\Ii_\Gamma} = 1_{\Lambda^\infty} \times \omega$ by construction, so it remains
to calculate the action $\tilde{r}^\sigma$ of $\Gg_\Lambda$ on $\widehat{Z}_\omega$
described in Lemma~\ref{lem:tilde-r cocycle}.

For this, we first claim that $\sigma_{c_\phi}$ satisfies
\begin{equation}\label{eq:cocycle formula}
\sigma_{c_\phi}((g, m), (h,n))
    = \big(\overline{\phi(\mu_g)}\phi(\mu_{gh})\big)^m\big(\overline{\phi(\nu_h)}\phi(\nu_{gh})\big)^n.
\end{equation}
To see this, choose $\alpha, \beta, \gamma \in \Lambda$ and $y \in \Lambda^\infty$
satisfying~\eqref{eq:alphabetagamma}. Then there exist $m_\alpha, m_\beta$ and $m_\gamma
\in \NN^l$ such that $\tilde\alpha = (\alpha, m_\alpha)$, $\tilde\beta = (\beta,
m_\beta)$ and $\tilde\gamma = (\gamma, m_\gamma)$ satisfy equations
\eqref{eq:alphabetagamma} with respect to the $\tilde\mu$'s and $\tilde\nu$'s. So
\begin{align}
\sigma_{c_\phi}((g, m), (h,n))
    &= c_\phi(\tilde\mu_g, \tilde\alpha)\overline{c_\phi}(\tilde\nu_g,\tilde\alpha)
         c_\phi(\tilde\mu_h, \tilde\beta)\overline{c_\phi}(\tilde\nu_h,\tilde\beta)
         \overline{c_\phi}(\tilde\mu_{gh}, \tilde\gamma)c_\phi(\tilde\nu_{gh},\tilde\gamma)\nonumber\\
    &= \phi(\alpha)^{m^+} \overline{\phi(\alpha)}^{m^-} \phi(\beta)^{n^+} \overline{\phi(\beta)}^{n^-}
        \overline{\phi(\gamma)}^{(m+n)^+}\phi(\gamma)^{(m+n)^-} \nonumber\\
    &= \big(\phi(\alpha)\overline{\phi(\gamma)}\big)^m \big(\phi(\beta)\overline{\phi(\gamma)}\big)^n.
        \label{eq:cocycle calculation}
\end{align}
We have $\phi(\mu_g)\phi(\alpha) = \phi(\mu_g\alpha) = \phi(\mu_{gh}\gamma) =
\phi(\mu_{gh})\phi(\gamma)$, and rearranging gives $\phi(\alpha)\overline{\phi(\gamma)} =
\overline{\phi(\mu_g)}\phi(\mu_{gh})$, and similarly, $\phi(\beta)\overline{\phi(\gamma)}
= \overline{\phi(\nu_h)}\phi(\nu_{gh})$. Substituting this into~\eqref{eq:cocycle
calculation}, we obtain~\eqref{eq:cocycle formula}.

Now, fix $p \in Z_\omega$ and $(g,n) = ((x,m,y),n) \in \Gg_\Gamma$.
Using~\eqref{eq:cocycle formula} at the second equality, we calculate:
\begin{align*}
r&^\sigma_{(g,n)}(p)\\
    &=\sigma\big((g,n), ((y,0,y),p)\big)\sigma\big((g, n + p),(g^{-1}, -n)\big)\overline{\sigma\big((g,n), (g^{-1},-n)\big)}\\
    &=\Big[\big(\overline{\phi(\mu_g)}\phi(\mu_{g})\big)^n\big(\overline{\phi(\nu_{s(g)})}\phi(\nu_{g})\big)^p\Big]
        \Big[\big(\overline{\phi(\mu_g)}\phi(\mu_{gg^{-1}})\big)^{n+p}\big(\overline{\phi(\nu_{g^{-1}})}\phi(\nu_{gg^{-1}})\big)^{-n}\Big]\\
    &\qquad\qquad \Big[\overline{\big(\overline{\phi(\mu_g)}\phi(\mu_{gg^{-1}})\big)^n
        \big(\overline{\phi(\nu_{g^{-1}})}\phi(\nu_{gg^{-1}})\big)^{-n}}\hskip1pt\Big] \omega(n,p)\omega(n+p,-n)\overline{\omega(n,-n)} \\
    &= \phi(\nu_g)^p \overline{\phi(\mu_g)}^{n+p}\phi(\nu_{g^{-1}})^n \phi(\mu_g)^n \phi(\nu_{g^{-1}})^{-n}
        \omega(n,p)\overline{\omega(p,n)} \\
    &= \big(\overline{\phi(\mu_g)}\phi(\nu_g)\big)^p (\omega\omega^*)(n,p).
\end{align*}
Since $p \in Z_\omega$, we have $(\omega\omega^*)(n,p) = 1$, and so $r^\sigma_{(g,n)}(p)
= \overline{\tilde{\phi}(g)}^p$. Since $(g,n) \mapsto g$ induces an isomorphism
$\Hh_\Gamma \cong \Gg_\Lambda$, we deduce that $\tilde{r}^\sigma_g(p) =
\overline{\tilde{\phi}(g)}^p$ for $g \in \Gg_\Lambda$ and $p \in Z_\omega$.
So~\eqref{eq:theta formula} follows from the definition of $\theta$.
\end{proof}

\begin{proof}[Proof of part (\ref{three}) of Theorem~\ref{thm:easytwistsimplicity}]
As observed in Remark~\ref{rmk:observant reader}, both conditions appearing in
statement~(\ref{three}) imply that $\Lambda$ is cofinal, so we may assume that this is
the case. The cocycle $\sigma$ of Proposition~\ref{prp:action formula} satisfies the
hypotheses of Theorem~\ref{thm:main simple}. So Theorem~\ref{thm:main simple} combined
with Proposition~\ref{prp:action formula} says that $C^*(\Gamma, c_{\phi,\omega})$ is
simple if and only if the action $\theta$ of $\Gg_\Lambda$ on $\Lambda^\infty \times
Z_\omega$ described in~\eqref{eq:theta formula} is minimal. The condition described in
the statement of Theorem~\ref{thm:easytwistsimplicity}(\ref{three}) is precisely the
statement that $\Gg_\Lambda \cdot (x,1)$, the orbit of $(x,1)$ under $\theta$, is dense
for all $x$. Since $\Gg_\Lambda \cdot (x,z) = \{ (y,zw) : (y,w) \in \Gg_\Lambda \cdot
(x,1) \}$, the result follows.
\end{proof}

As a special case of Theorem~\ref{thm:easytwistsimplicity}, we obtain the following.

\begin{cor}\label{cor:1graphspecial}
Let $E$ be a strongly connected graph which is not a simple cycle. Let $\phi$ be a
function from $E^1$ to $\TT$. There is an action $\beta : \ZZ \to \Aut(C^*(E))$ such that
$\beta_n(s_e) = \phi(e)^ns_e$ for all $e \in E^1$ and $n \in \ZZ$. If there is a cycle
$\mu = \mu_1 \cdots \mu_n$ such that $\prod_i \phi(\mu_i) = e^{2\pi i \theta}$ for some
$\theta \not\in\QQ$, then $C^*(E) \times_\beta \ZZ$ is simple.
\end{cor}
\begin{proof}
We will apply Theorem~\ref{thm:easytwistsimplicity} to the $1$-graph $E^*$ and the
extension of $\phi$ to a $1$-cocycle $\phi : E^* \to \TT$. Fix $x \in E^\infty$. By
Theorem~\ref{thm:easytwistsimplicity}, it suffices to show that
\[
\overline{\{(r(\gamma), \tilde{\phi}(\gamma)) : \gamma \in \Gg_E, s(\gamma) = x\}}
    = E^\infty \times \TT.
\]
Choose a basic open set $Z(\eta)$ of $E^\infty$ and $z \in \TT$. It suffices to show that
there are elements $\gamma_n$ with $s(\gamma_n) = x$ and $r(\gamma_n) \in Z(\eta)$ such
that $\tilde{\phi}(\gamma_n) \to z$.

Since $E$ is strongly connected, there exist $\alpha \in s(\eta)E^* r(\mu)$ and $\beta
\in r(\mu) E^* r(x)$. The elements $\gamma_n := (\eta\alpha\mu^n\beta x,
d(\eta\alpha\mu^n\beta), x) \in \Gg_E$ satisfy $s(\gamma_n) = x$, $r(\gamma_n) \in
Z(\eta)$, and $\tilde{\phi}(\gamma_n) := e^{2\pi i n \theta} \tilde{\phi}(\gamma_0)$.
Since $\theta$ is irrational, these values are dense in $\TT$.
\end{proof}

We conclude the section by giving a version of Theorem~\ref{thm:easytwistsimplicity}
whose statement does not require groupoid technology. Let $\Lambda$ be a row-finite
$k$-graph with no sources, and take $\phi \in \Zcat1(\Lambda, \TT^l)$ and $\omega$ a
bicharacter of $\ZZ^l$. Regard each $\phi(\lambda)$ as a character of $\ZZ^l$ so that
$\phi(\lambda)|_{Z_\omega}$ is a character of $Z_\omega$. Then $\Lambda$ acts on
$\Lambda^\infty \times \widehat{Z}_\omega$ by partial homeomorphisms $\{\vartheta_\lambda
: \lambda \in \Lambda\}$ where each $\dom(\vartheta_\lambda) = Z(s(\lambda)) \times
\widehat{Z}_\omega$, and
\[
\vartheta_\lambda(x,\chi) = (\lambda x, \phi(\lambda)|_{Z_\omega} \chi)
    \text{ for } (x,\chi) \in Z(s(\lambda)) \times \widehat{Z}_\omega.
\]

Define a relation $\sim$ on $\Lambda^\infty \times \widehat{Z}_\omega$ by
\begin{align*}
(x,\chi) \sim (x',\chi') &\text{ iff there exist }
(y,\rho) \in \Lambda^\infty \times \widehat{Z}_\omega
    \text{ and }\lambda, \mu \in \Lambda\text{ such that} \\
 &s(\lambda) = s(\mu) = r(y), \vartheta_\lambda(y,\rho) = (x,\chi)
 \text{ and } \vartheta_\mu(y, \rho) = (x',\chi').
\end{align*}
Then $\sim$ is an equivalence relation; indeed, it is the equivalence relation induced by
$\vartheta$. We write $[x,\chi]$ for the equivalence class of $(x,\chi)$ under $\sim$.

\begin{cor}\label{cor:justonemore}
Let $\Lambda$ be an aperiodic row-finite $k$-graph with no sources and take $\phi \in
\Zcat1 (\Lambda, \TT^l)$. Let $c_{\phi, \omega} \in \Zcat2  (\Lambda \times T_l, \TT)$ be
as defined in \eqref{eq:phicomega}. Then $C^*(\Lambda \times T_l, \phi_{c, \omega})$ is
simple if and only if $[x,1]$ is dense in $\Lambda^\infty \times \widehat{Z}_\omega$ for
all $x \in \Lambda^\infty$.
\end{cor}
\begin{proof}
It is routine to check that $[x,1]$ is equal to the set $\{(r(\lambda) ,
\tilde{\phi}(\gamma)|_{Z_\omega} : \gamma \in (\Gg_\Lambda)_x\}$ of
Theorem~\ref{thm:easytwistsimplicity} part~(\ref{three}).
\end{proof}

\section{Examples}\label{sec:examples}

First we discuss an example of a nondegenerate 2-cocycle on $\ZZ^2$ pulled back over the
degree map to a cocycle on a cofinal $2$-graph for which the twisted $C^*$-algebra is not
simple.
This shows that the hypothesis of \cite[Theorem~7.1]{SWW} that $c|_{\Per(\Lambda)}$
is nondegenerate cannot be relaxed to the weaker assumption that $c$ is nondegenerate
as a cocycle on $\ZZ^k$.

\begin{example}\label{ex:nonsimple}
Let $\theta \in \RR \setminus \QQ$, and define a $1$-cocycle $\phi$ on $\Omega_1$ by
$\phi(m,n) = e^{2\pi i (n-m)\theta}$. As in Section~\ref{sec:CP}, define $c_\phi \in
\Zcat2(\Omega_1 \times T_1, \TT)$ by $c_\phi((\alpha,m),(\beta,n)) = \phi(\beta)^m$. Then
$c_\phi = \varpi \circ d$ where $d : \Omega_1 \times T_1 \to \NN^2$ is the degree map,
and $\varpi \in Z^2(\ZZ^2, \TT)$ is the cocycle $\varpi(m,n) = e^{2\pi i \theta m_2n_1}$.
Since $\theta$ is irrational, the antisymmetric bicharacter associated to $\varpi$ is
nondegenerate.  Note that $\Per(\Lambda) \cong \ZZ$ and so the restriction of $\varpi$ to
$\Per(\Lambda)$ is degenerate.

We have $\Gg_{\Omega_1} \cong \NN \times \NN$ as an equivalence relation. So
$\Gg_{\Omega_1}^{(0)} \times \TT \cong \NN \times \TT$, and under this identification,
the action $\theta$ described in Proposition~\ref{prp:action formula} boils down to
\[
\theta_{(m,n)}(n, z) = (m, e^{2\pi i \theta (n-m)}z).
\]
This is not minimal because each orbit intersects each $\{n\} \times \TT$ in a singleton.
So the resulting twisted $C^*$-algebra is not simple.
\end{example}

We present a second example showing that the same pulled-back cocycle as used in
Example~\ref{ex:nonsimple} can yield a simple $C^*$-algebra if connectivity in the
underlying $1$-graph $\Lambda$ is slightly more complicated than in $\Omega_1$.

\begin{example}
Let $\Lambda$ be a strongly connected 1-graph with at least one edge and suppose that
$\Lambda$ is not a simple cycle. Then $\Lambda$ is cofinal and aperiodic (see, for
example, \cite{KPR}) and contains a nontrivial cycle $\lambda$. A simple example is
$\Lambda = B_2$ the bouquet of 2 loops, so that $C^*(\Lambda) = \Oo_2$.

Take $\theta \in \RR \setminus \QQ$, and define a $1$-cocycle $\phi$ on $\Lambda$ by
$\phi(\lambda) = e^{2\pi i d(\lambda)\theta}$. As in the preceding section, construct
$c_\phi \in \Zcat2(\Lambda \times T_1, \TT)$ by $c_\phi((\alpha,m),(\beta,n)) =
\phi(\beta)^m$. Since $\theta \in \RR \setminus \QQ$, the paths $\mu = \lambda$ and $\nu
= r(\lambda)$ satisfy $r(\mu) = r(\nu)$ and $s(\mu) = s(\nu)$, and
$\phi(\mu)\overline{\phi(\nu)} = e^{2\pi i d(\lambda)\theta}$ where $d(\lambda)\theta$ is
irrational. So Corollary~\ref{cor:1graphspecial} shows that $C^*(\Lambda \times T_1,
c_\phi)$ is simple.
\end{example}

Our next example shows that it is possible for the bicharacter $\omega$ of
$\Per(\Lambda)$ determined by $c$ as in Proposition~\ref{prp:constant in Ii} to be
degenerate and still have $C^*(\Lambda, c)$ simple.

\begin{example}
Let $\Lambda$ be the $1$-graph with vertex $v$ and edges $\{e , f \}$, i.e. the bouquet
$B_2$ of two loops. Fix $\theta \in \RR \setminus \QQ$ and define $\phi : \Lambda^1 \to
\TT$ by $\phi (e) = 1$ and $\phi (f) = e^{2\pi i\theta}$. As pointed out in
\cite[Page~917]{E} one can check that the action $\beta = \beta^\phi$ of $\ZZ$ on $C^* (
\Lambda )$ as given in part~(\ref{one}) of Theorem~\ref{thm:easytwistsimplicity} is
outer. Hence by \cite[Theorem 3.1]{K} $C^* ( \Lambda ) \times_{\beta} \ZZ$ is simple (and
purely infinite by \cite[Lemma 10]{KK}). Since $\Per(\Lambda \times T_1) \cong \ZZ$ the
bicharacter $\omega$ of Proposition~\ref{prp:constant in Ii} is degenerate. Note that we
could also deduce simplicity of $C^*(\Lambda) \times_{\beta} \ZZ$ from
Corollary~\ref{cor:justonemore}.
\end{example}

The next example illustrates that the interaction between the action $\theta$, the group
$\Per(\Lambda)$ and the subgroup $Z_\omega$ can be fairly complicated.

\begin{example}
Consider the $4$-graph $\Lambda = B_2 \times T_3$, where $B_2$ is the bouquet of two
loops as in the preceding example, and $T_3$ denotes $\NN^3$ regarded as a $3$-graph.
Choose irrational numbers $\theta$ and $\rho$. Choose any $\phi : B^1_2 \to \TT^3$ such
that $\phi(e)_1 = 1$ and $\phi(f)_1 = e^{2\pi i \theta}$, and let $c_\phi \in \Zcat2(B_2
\times T_1, \TT^3)$ be as in~\eqref{eq:phic}. Define a bicharacter $\omega$ of $\ZZ^3$ by
$\omega(p,q) = e^{2\pi i \rho q_2p_3}$. Let $c_{\phi,\omega}$ be the cocycle
$c_{\phi,\omega}((\alpha, m, p),(\beta, n, q)) = c_\phi((\alpha,m),(\beta,n))
\omega(p,q)$ of~\eqref{eq:phicomega}. We have $\Per(\Lambda) = \{0\} \times \ZZ^3 \subset
\ZZ^4$. It is routine to check using Lemma~\ref{lem:compute omega} that $Z_\omega = \{0\}
\times \ZZ \times \{0\} \subseteq \Per(\Lambda)$, so is a proper nontrivial subgroup of
$\Per(\Lambda)$. The argument of Corollary~\ref{cor:1graphspecial} shows that
$\{(r(\gamma), \tilde{\phi}(\gamma)|_{Z_\omega}) : \gamma \in (\Gg_{B_2})_x\}$ is dense
in $B_2^\infty \times \widehat{Z}_\omega$ for all $x \in B_2^\infty$. So
Theorem~\ref{thm:easytwistsimplicity}(\ref{three}) shows that $C^*(\Lambda,
c_{\phi,\omega})$ is simple. Note that $\{(r(\gamma), \tilde{\phi}(\gamma)) : \gamma \in
(\Gg_{B_2})_x\}$ may not be dense in $B_2^\infty \times \Per(\Lambda)\widehat{\;}$ for
any $x$; for example if $\phi(e)_3 = \phi(f)_3 = 1$.
\end{example}

Finally, we present an example which is not directly related to our simplicity theorems
here, but is of independent interest. Corollary~8.2 of \cite{KPS4} implies that if
$C^*(\Lambda)$ is simple, then each $C^*(\Lambda, c)$ is simple. Moreover, Theorem~5.4 of
\cite{KPS5} shows that $C^*(\Lambda)$ and $C^*(\Lambda, c)$ very frequently (perhaps
always) have the same $K$-theory. From this and the Kirchberg-Phillips theorem
\cite{kirchpureinf, Phillips:Doc00} we deduce that if $C^*(\Lambda)$ is purely infinite
and simple, then all its twisted $C^*$-algebras for cocycles that lift to $\RR$-valued
cocycles coincide. On the other hand, the twisted $C^*$-algebras of $T_2$, for example,
are rotation algebras and so different choices of cocycle yield different twisted
$C^*$-algebras; but in this instance the untwisted algebra is not simple. We were led to
ask whether there exists a $k$-graph whose twisted $C^*$-algebras (including the
untwisted one) are all simple but are not all identical. We present a $3$-graph with this
property.

\begin{example}
Consider the $3$-coloured graph with vertices $\{v_n : n = 1, 2, \dots\}$, blue edges
$\{e_{n,i} : n \in \NN\setminus\{0\}, i \in \ZZ/n^2\ZZ\}$, red edges $\{f_n : n \in
\NN\setminus\{0\}\}$ and green edges $\{g_n : n \in \NN\setminus\{0\}\}$, with
$r(e_{n,i}) = r(f_n) = s(f_n) = r(g_n) = s(g_n) = v_n$ and $s(e_{n,i}) = v_{n+1}$ for all
$n,i$. The graph can be drawn as follows. (For those with monochrome printers: the solid
edges are blue, the dashed edges are red, and the edges drawn in a dash-dot-dash pattern
are green.)
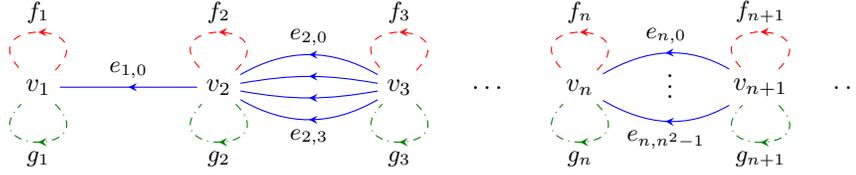
\[
\begin{tikzpicture}[scale=1.2,>=stealth, decoration={markings, mark=at position 0.5 with {\arrow{>}}}]
    \node (v1) at (0,0) {$v_1$};
    \node (v2) at (2,0) {$v_2$};
    \node (v25) at (4,0) {$v_3$};
    \node at (5,0) {\dots};
    \node (v3) at (6,0) {$v_n$};
    \node (v4) at (8,0) {$v_{n+1}$};
    \node at (9,0) {\dots};
    \draw[blue, postaction = decorate] (v2) to node[pos=0.5, above, black] {\small$e_{1,0}$} (v1);
    \draw[blue, postaction = decorate, in=30, out=150] (v25) to node[pos=0.5, above, black] {\small$e_{2,0}$} (v2);
    \draw[blue, postaction = decorate, in=10, out=170] (v25) to (v2);
    \draw[blue, postaction = decorate, in=350, out=190] (v25) to (v2);
    \draw[blue, postaction = decorate, in=330, out=210] (v25) to node[pos=0.5, below, black] {\small$e_{2, 3}$} (v2);
    \draw[blue, postaction = decorate, in=30, out=150] (v4) to node[pos=0.5, above, black] {\small$e_{n,0}$} (v3);
    \draw[blue, postaction = decorate, in=330, out=210] (v4) to node[pos=0.5, below, black] {\small$e_{n, n^2-1}$} (v3);
    \node at (7,0.1) {$\vdots$};
    \draw[red, dashed, postaction=decorate] (v1) .. controls +(0.75,0.75) and +(-0.75,0.75) .. (v1) node[pos=0.5, above, black] {\small$f_1$};
    \draw[red, dashed, postaction=decorate] (v2) .. controls +(0.75,0.75) and +(-0.75,0.75) .. (v2) node[pos=0.5, above, black] {\small$f_2$};
    \draw[red, dashed, postaction=decorate] (v25) .. controls +(0.75,0.75) and +(-0.75,0.75) .. (v25) node[pos=0.5, above, black] {\small$f_3$};
    \draw[red, dashed, postaction=decorate] (v3) .. controls +(0.75,0.75) and +(-0.75,0.75) .. (v3) node[pos=0.5, above, black] {\small$f_n$};
    \draw[red, dashed, postaction=decorate] (v4) .. controls +(0.75,0.75) and +(-0.75,0.75) .. (v4) node[pos=0.5, above, black] {\small$f_{n+1}$};
    \draw[green!50!black, dashdotted, postaction=decorate] (v1) .. controls +(0.75,-0.75) and +(-0.75,-0.75) .. (v1) node[pos=0.5, below, black] {\small$g_1$};
    \draw[green!50!black, dashdotted, postaction=decorate] (v2) .. controls +(0.75,-0.75) and +(-0.75,-0.75) .. (v2) node[pos=0.5, below, black] {\small$g_2$};
    \draw[green!50!black, dashdotted, postaction=decorate] (v25) .. controls +(0.75,-0.75) and +(-0.75,-0.75) .. (v25) node[pos=0.5, below, black] {\small$g_3$};
    \draw[green!50!black, dashdotted, postaction=decorate] (v3) .. controls +(0.75,-0.75) and +(-0.75,-0.75) .. (v3) node[pos=0.5, below, black] {\small$g_n$};
    \draw[green!50!black, dashdotted, postaction=decorate] (v4) .. controls +(0.75,-0.75) and +(-0.75,-0.75) .. (v4) node[pos=0.5, below, black] {\small$g_{n+1}$};
\end{tikzpicture}
\]
Specify commuting squares by $f_n e_{n,i} = e_{n, i+1} f_{n+1}$, $g_n e_{n,i} = e_{n,
i+n} g_{n+1}$ and $f_n g_n = g_n f_n$ (addition in the subscripts of the $e_{n,i}$ takes
place in the cyclic group $\ZZ/n^2\ZZ$). It is easy to check that this is a complete and
associative collection of commuting squares as in \cite{HRSW}, and so determines a
3-graph $\Lambda$, in which each $d(e_{n,i}) = (1,0,0)$, each $d(f_n) = (0,1,0)$ and
$d(g_n) = (0,0,1)$. It is clear that $\Lambda$ is cofinal. We claim that it is also
aperiodic. For this, \cite[Remark~3.2 and Theorem~3.4]{LS} imply that it is enough to
show that if $\mu,\nu \in \Lambda$ satisfy
\begin{equation}\label{eq:munu}
r(\mu) = r(\nu),\quad s(\mu) = s(\nu)
    \quad\text{ and }\quad d(\mu) \wedge d(\nu) = 0
\end{equation}
and if $\mu\alpha \Lambda \cap \nu\alpha \Lambda \not= \emptyset$ for all $\alpha \in
s(\mu)\Lambda^0$ then $\mu = \nu = r(\mu)$. Fix $\mu,\nu$ satisfying~\eqref{eq:munu}.
Then $\mu = f_m^p$ and $\nu = g_m^q$ (or vice versa) for some $p,q \in \NN$. Let $n =
\max\{p, q\}$ so that $m+n > p,q$, and let $\alpha = e_{m,0}e_{m+1, 0} \cdots\, e_{m+n,0}$.
The factorisation property implies that $\mu\alpha$ has the form $e_{m, i_m} e_{m+1,
i_{m+1}} \cdots\, e_{m+n, p} f_{m+n+1}^p$ and $\nu\alpha$ has the form $e_{m, j_m} e_{m+1,
j_{m+1}} \cdots\, e_{m+n, q(m+n)} g_{m+n+1}^q$. So $\mu\alpha \Lambda \cap \nu\alpha \Lambda
\not= \emptyset$ forces $p = q(m+n)$ in $\ZZ/(m+n)^2\ZZ$. Since $m+n \ge p,q$, this
forces $p = q = 0$ so that $\mu = \nu = v_m$ as required. So $\Lambda$ is aperiodic as
claimed.

It now follows from \cite[Corollary~8.2]{KPS4} that $C^*(\Lambda, c)$ is simple for every
$c \in \Zcat2(\Lambda, \TT)$. Let $\theta \in [0, 1)$ and define $c_\theta \in
\Zcat2(\Lambda, \TT)$ by $c_\theta (\mu,\nu) = e^{2\pi i \theta d(\mu)_3 d(\nu)_2}$. We
show that $C^*(\Lambda, c_\theta)$ and $C^*(\Lambda, c_\rho)$ are nonisomorphic whenever
$\theta$ and $\rho$ are rationally independent\footnote{In an earlier version of the
paper, we proved only that $C^*(\Lambda,\theta)$ and $C^*(\Lambda)$ are nonisomorphic
when $\theta$ is irrational. We thank the anonymous referee for suggesting that we expand
our analysis to encompass the relationship between these algebras for different
irrational values of $\theta$. We believe that the end product is cleaner and more
informative even in in the situation we had originally considered.}.

For each $n = 1, 2, \dots$ the Cuntz-Krieger relations give $s_{f_n} s^*_{f_n} =
s^*_{f_n} s_{f_n} = s_{g_n} s^*_{g_n} = s^*_{g_n} s_{g_n} = p_{v_n}$, so $s_{f_n}$ and
$s_{g_n}$ are unitaries in $C^*(\{s_{f_n}, s_{g_n}\})$. The definition of $c_\theta$
gives $s_{g_n} s_{f_n} = e^{2\pi i \theta} s_{f_n} s_{g_n}$. This is the defining
relation for the rotation algebra $A_\theta$ and so $C^*(\{s_{f_n}, s_{g_n}\}) \cong
A_\theta$. We observe that we can express the corner $p_{v_1} C^*(\Lambda, c_\theta)
p_{v_1}$ as the direct limit of the $C^*$-algebras $p_{v_1} C^*(\Lambda_n, c_\theta)
p_{v_1}$, where $\Lambda_n$ is the locally-convex 3-graph $\{v_1, \dots, v_n\} \Lambda
\{v_1, \dots, v_n\}$. Each of these subalgebras is canonically isomorphic to
$M_{q_n}(A_\theta)$, where $q_n = \prod^{n-1}_{i=1} i^2$ (note $q_1 = q_2 = 1$), so its
$K_0$-group is isomorphic to $\ZZ^2$.

We claim that if  $\theta$ is irrational, then the map on $K_0$ induced by the inclusion
map
\[
p_{v_1} C^*(\Lambda_n, c_\theta) p_{v_1} \hookrightarrow p_{v_1} C^*(\Lambda_{n+1}, c_\theta) p_{v_1},
\]
is given by multiplication by $n^2$. This is clear by direct computation using the
Cuntz--Krieger relation for the class $[p_{v_1}]$ of the identity, and therefore for the
class of a minimal projection in $M_{q_n}(\CC 1) \subseteq M_{q_n}(A_\theta)$.
The other generating $K_0$-class is that of the matrix $p_n \in
A_\theta \subseteq M_{q_n}(A_\theta)$ with $(1,1)$-entry equal to the Powers-Rieffel
projection and all other entries zero. The trace on $M_{q_{n+1}}(A_\theta)$ carries
$p_{n+1}$ to $\theta/q_{n+1}$, while its restriction to the image of $M_{q_n}(A_\theta)$
must agree with the trace on $M_{q_n}(A_\theta)$ so carries the image of $p_n$ to
$\theta/q_n$; so $[\iota(p_n)]_0 = (q_{n+1}/q_n)[p_{n+1}]_0 = n^2 [p_{n+1}]_0$.

We deduce from the continuity of $K$-theory as a functor into the category of ordered
abelian groups (see for example \cite[Theorem~6.3.2]{RLL}) that when $\theta$ is
irrational, $K_0(C^*(\Lambda, c_\theta))$ is isomorphic as an ordered abelian group to
$\QQ + \theta \QQ$ with the order inherited from $\RR$, and hence isomorphic as a group
to $\QQ^2$. Since each $c_\theta$ is of exponential form, Theorem~5.4 of \cite{KPS5} (see
also \cite{Gillaspy}) shows that $K_0(C^*(\Lambda, c_\theta)) \cong \QQ^2$ for all
$\theta$.

Since the tracial-state space of each $p_{v_1} C^*(\Lambda_n) p_{v_1}$ is compact and
nonempty and $p_{v_1} C^*(\Lambda_n) p_{v_1}$ are nested, compactness shows that $p_{v_1}
C^*(\Lambda, c_\theta) p_{v_1}$ admits a trace $\tau$ (or one can directly construct a
trace using the approach of \cite[Proposition~3.8]{PaskRennieEtAl:kth08}). The functional
$\tau_*$ on $K_0(p_{v_1} C^*(\Lambda, c_\theta) p_{v_1})$ induced by any trace $\tau$ has
range $\bigcup \tau_*|_{K_0(p_{v_1} C^*(\Lambda_n, c_\theta) p_{v_1})}$. Since each
$p_{v_1} C^*(\Lambda_n, c_\theta) p_{v_1} \cong M_{q_n}(A_\theta)$, uniqueness of the map
on $K_0(A_\theta)$ induced by a trace on $A_\theta$ (see \cite[Lemma 2.3]{El}) implies
that every trace on $p_{v_1} C^*(\Lambda, c_\theta) p_{v_1}$ induces the same functional
$\tau_* : K_0(p_{v_1} C^*(\Lambda_n, c_\theta) p_{v_1}) \to \RR$, and that this $\tau_*$
has range $\bigcup_n(\ZZ + \theta \ZZ)/q_n = \QQ + \theta \QQ$.

It follows that the Morita-equivalence classes (and so in particular the isomorphism
classes) of $C^*(\Lambda, c_\theta)$ and $C^*(\Lambda, c_\rho)$ are distinct for
rationally independent $\theta$ and $\rho$. In particular if $\theta$ is irrational then
$C^*(\Lambda, c_\theta)$ and $C^*(\Lambda) = C^*(\Lambda, c_0)$ are not isomorphic.
\end{example}

\end{document}